\documentclass[11pt,reqno]{amsart}
\usepackage{amssymb}
\usepackage{amsthm}
\usepackage{eucal}
\usepackage[latin1]{inputenc}

\newcommand{\R}{\mathbb R}

\newcommand{\Z}{\mathbb Z}

\newcommand{\hil}{\mathcal{H}}
\newcommand{\ma}{\mathcal{A}}
\newcommand{\mb}{\mathcal{B}}
\newcommand{\md}{\mathcal{D}}

\newcommand{\sgn}{\text{sgn}}

\newcommand{\wnt}{w^{\theta}_N}
\newcommand{\e}{\epsilon}

\newcommand{\jd}{\rangle}

\newcommand{\ji}{\langle}

\usepackage{color}

\newtheorem{theorem}{Theorem}[section]
\newtheorem{proposition}[theorem]{Proposition}
\newtheorem{remark}[theorem]{Remark}
\newtheorem{lemma}[theorem]{Lemma}

\newtheorem{claim}[theorem]{Claim}

\newtheorem{thmx}{Theorem}

\author[]{}

\title[The  ${\text f}$BBM equation]{On the persistence properties  for the fractionary BBM  equation with low dispersion in weighted Sobolev spaces}

\author[G. Fonseca]{Germ\'an Fonseca}
\address{Departamento de Matem\'aticas, Universidad Nacional de Colombia Carrera 45 No. 26-85, Edificio Uriel Guti\'errez Bogot\'a D.C., Colombia}
\curraddr{}
\email{gefonsecab@unal.edu.co}

\author[O. Ria\~no]{Oscar Ria\~no}
\address{Departamento de Matem\'aticas, Universidad Nacional de Colombia Carrera 45 No. 26-85, Edificio Uriel Guti\'errez Bogot\'a D.C., Colombiaa}
\curraddr{}
\email{ogrianoc@unal.edu.co}

\author[G. Rodriguez-Blanco]{Guillermo Rodriguez-Blanco}
\address{Departamento de Matem\'aticas, Universidad Nacional de Colombia Carrera 45 No. 26-85, Edificio Uriel Guti\'errez Bogot\'a D.C., Colombia}
\curraddr{}
\email{grodriguezb@unal.edu.co}

\begin{document}
\dedicatory{In memory of Rafael I\'orio}
\begin{abstract} We consider the initial value problem associated to the low dispersion fractionary Benjamin-Bona-Mahony equation, fBBM. Our aim is to establish local persistence results in weighted Sobolev spaces and to obtain unique continuation results that imply that those results above are sharp. Hence, arbitrary polynomial type decay is not preserved by the fBBM flow.

\end{abstract}

\maketitle

\section{Introduction.}
In this work, we shall study  the initial value problem (IVP) for the fractionary Benjamin-Bona-Mahony (fBBM) equation
\begin{equation}\label{fBBM}
\begin{cases}
\partial_t u+\partial_x u+ D^{\alpha}\partial_{t} u+\partial_x (u^2) = 0, \qquad t, x\in \R,\;\;\;\; \\
u(x,0) = \varphi(x),
\end{cases}
\end{equation}
where  $\alpha\in(0,1)$ and $D^{\alpha}=(-\partial_x^2)^{\frac{\alpha}{2}}$ denotes the fractionary derivative of order $\alpha$ in the $x$ direction. 

This equation is a generalization of the famous BBM $(\alpha=2)$   and rBO $(\alpha=1)$ equations which were originally introduced by Benjamin, Bona, and Mahony in \cite{BenjaminBonaMahony1972}. They investigated these equations as regularized versions of the  Korteweg-de Vries equation (KdV) and Benjamin-Ono equation (BO) respectively in the context of small amplitude long waves in nonlinear dispersive systems. We emphasize that they have also been studied from a mathematical and physical point of view in many other contexts as well (see {\it e.g.,} \cite{Peregrine1966,MeissHorton1982,An,KalischBona2000,FoRoSa} and references therein).


The study of competition between dispersion and nonlinearity in the dynamics of solutions for nonlinear dispersive partial differential equations has led to several studies. The problem is generally handled by fixing the dispersion and increasing the nonlinearity. Alternatively,  the use of fractional derivatives has become popular, as these operators allow testing different dispersive effects for a fixed nonlinearity. Following this strategy, equation fBBM appears as a fractional version of BBM equation, in which variations on $\alpha>0$ offer flexibility to study the behavior of solutions under different dispersion regimes. In particular, this paper focuses on the \emph{lower dispersion setting} $\alpha\in (0,1)$, where it is worth mentioning that the numerical simulation of \cite{KleinSaut2015} suggests that a finite type of blow-up may happen for solutions of fBBM when $0<\alpha\leq \frac{1}{3}$,  but not when $\alpha>\frac{1}{3}$ (where solutions may be global). To a certain extent, this shows that cases $\alpha\in(0,1)$ in fBBM are interesting to have a complete overview of the instances where the dispersive effects are stronger or weaker than the nonlinear ones. 


We also remark that another model with variable dispersion that has been extensively studied is the \emph{fractional Korteweg-de Vries (fKdV)} equation
\begin{equation}\label{fKdV}
  \partial_t u-\partial_x D^{\alpha}u+\partial_x(u^2)=0, \qquad x,t\in \mathbb{R}. 
\end{equation}
Note that $\alpha=2$ in fKdV is the KdV equation, and $\alpha=1$ coincides with the BO equation. For different studies of the fKdV equation with $\alpha\in [-1,2)$, $\alpha \neq 0, 1$, we refer to \cite{Angulo2018,HerIonescuKenigKoch2010,KenigPilodPonceVega2020,GFFLGP1,Ri,HunterIfrimTataruWong2015,KenigMartelRobbiano2011,LiPiSa,Argenis2020,MolinetPilodVento2013,CunhaRiano2022} and references therein.

Regarding invariants, in contrast with fKdV, the equation fBBM is not invariant under any rescaling $u_{\lambda}(x,t)=\lambda^{a}u(\lambda^{b}x, \lambda^{c} t)$, $\lambda>0$. Nevertheless, fBBM formally conserves the \emph{Energy}
\begin{equation}\label{Energy}
E[u(t)]=\int_{\mathbb{R}} \big((D^{\frac{\alpha}{2}}u(x,t))^2+ (u(x,t))^2 \big)\, dx,     
\end{equation}
and the \emph{Hamiltonian}
\begin{equation}\label{Hamil}
H[u(t)]=\frac{1}{2}\int_{\mathbb{R}} \big((u(x,t))^2+ \frac{2}{3}(u(x,t))^3 \big)\, dx.     
\end{equation}

Concerning local well-posedness for the IVP \eqref{fBBM} (a Cauchy problem is well-posed in a given function space one has existence, uniqueness of solutions, and continuous dependence of the data-to-solution flow map), we recall that in \cite{LiPiSa} local well-posedness (LWP) for \eqref{fBBM} is proven in $H^s(\R)$ with $s> \max\{\frac32-\alpha,1\}$, see Theorem A below. For studies on the existence time of solutions of fBBM, see \cite{Nilsson2022}. 
\\ \\
In this work, we seek to study the following question: \emph{If given initial data with a certain spatial polynomial type decay, is there a solution of the respective Cauchy problem \eqref{fBBM} that preserves this same spatial decay}?
\\ \\
Given the generality of the previous question, its study is typically restricted to showing persistence of solutions in the Sobolev spaces
\begin{equation*}
  Z_{s,r}=Z_{s,r}(\mathbb{R}):=H^s(\mathbb{R})\cap L^2(|x|^{2r}\,dx),  
\end{equation*}
which are natural settings for dispersive equations, and in them, the regularity is measured through the parameter $s$, and the spatial decay with $r>0$. Consequently, the results of this article are divided into two parts. In the first one, we will give a positive result of persistence in weighted spaces, i.e., we show that given $\alpha\in (0,1)$, there exist $s_{\alpha}>0$, such that if $s>s_{\alpha}$, $0<r<\frac{3}{2}+\alpha$, then for any $\varphi \in Z_{s,r}$, there exists a local solution of \eqref{fBBM} with initial condition $\varphi$ in the class $C([0,T];Z_{s,r})$. In the following part, we use unique continuation principles (UCP), to show that the decay parameter $r=\frac{3}{2}+\alpha$ is in a sense optimal in the $L^2$ theory of solutions of \eqref{fBBM}. More precisely, we will show that for some initial conditions $\varphi$, the corresponding solutions $u$ of \eqref{fBBM} cannot be at two different times in the class $Z_{s,\frac{3}{2}+\alpha}$ unless it is
the null solutions. We emphasize that the above questions have been extensively studied for fKdV with $\alpha\geq -1$ (see \cite{Ri,FoPo,GFFLGP1,CunhaRiano2022}), but this work appears to be the first to study such questions for the fractional fBBM. We will also deduce that the behavior of solutions of fBBM in $Z_{s,r}$-spaces differs from that of the fKdV equation.


Our work is also motivated by the spatial decay of solitary wave solutions, which are special solutions of \eqref{fBBM}, traveling in time and preserving its shape.  The {\it traveling  waves} for the equation \eqref{fBBM} are of the form $u(x,t)=Q_c(x-ct)$, $c>0$, where $Q_c$ satisfies 
\begin{equation}\label{EQ:groundState}
(c-1)\,Q_c+cD^{\alpha}Q_c-Q^{2}_c=0.
\end{equation}
When $c>1$, writing
\begin{equation}\label{E:Qscaled}
\Psi=\frac{2}{(c-1)}Q_c\Big(\Big(\frac{c}{c-1}\Big)^{\frac{1}{\alpha}}\,  x\Big), \qquad x\in \mathbb{R},
\end{equation} 
one has that $\Psi$ solves
\begin{equation}\label{Greq}
  \Psi+D^{\alpha}\Psi-\frac{1}{2}\Psi^2=0.  
\end{equation}
The results in \cite{FraLenz2013,FraLenzSilv2016} show that for any $0<\alpha<1$ there exists a unique non-trivial, non-negative, even function $\Psi \in H^{\frac{\alpha}2}(\mathbb R)\cap C^{\infty}(\mathbb{R})$ solution of \eqref{Greq}, which satisfies
\begin{equation}\label{poldecayGS}
\frac{C_1}{1+|x|^{1+\alpha}} \leq \Psi(x) \leq \frac{C_2}{1+|x|^{1+\alpha}}
\end{equation}
for all $x \in \mathbb{R}$ with some constants $C_2 \geq C_1>0$ (which depend on $\alpha$, $\Psi$). Therefore, the connection established through the scaling \eqref{E:Qscaled} shows that for $c>1$, the solitary wave solutions of fBBM have a polynomial spatial decay, whose order is determined by the dispersion $\alpha$ as in \eqref{poldecayGS}. For this reason, it is natural to study questions on well-posedness in spaces $Z_{s,r}(\mathbb{R})$. Concerning stability issues of solitary waves, we refer to \cite{An,Angulo2018,KleinSaut2015,LinaresPilodSaut2015,Zeng2006}.


\subsection{Main results}

Before presenting our main results, we begin by recalling the following local well-posedness (LWP) result for solutions of \eqref{fBBM}. 

\begin{thmx}\label{theorem1}  
 The Cauchy problem \eqref{fBBM} is LWP in $H^s(\R)$, $s> \max\{\frac32-\alpha,1\}$. 
\end{thmx}

Theorem \ref{theorem1} is established in \cite[Theorem 4.12]{LiPiSa}, let us mention some comments.  Some estimates, for example \cite[(4.26)]{LiPiSa}, force that LWP result requires the condition $s>1$ for any $\alpha\in (0,1)$.  On the other hand, to prove uniqueness of solutions, it is only necessary to consider uniqueness in $H^{\frac{\alpha}{2}}(\mathbb{R})$, whose estimates can be deduced from integration by parts and using Sobolev embedding without going through fractional Leibniz's rule.

Using the local existence of solutions in $H^s(\mathbb{R})$ granted by the previous theorem, our first main result shows the persistence of solutions in the space $Z_{s,r}$.
       
\begin{theorem}\label{Main2} Let $\varphi \in Z_{s,r}(\mathbb{R})$ with $ s>s_{\alpha}$ and $0<r<\frac32+\alpha$, where $s_{\alpha} =1$ if $\alpha\in[\frac{1}{2},1)$ and $ s_{\alpha}=2-2\alpha$ if $\alpha\in (0,\frac{1}{2}).$  Then there  exist a time $T>0$ and a unique $u\in C([0,T];Z_{s,r})$ solution of the initial value problem \eqref{fBBM} with initial condition $\varphi$.
\end{theorem}

Theorem \ref{Main2} is proved using energy estimates, which are involved given the small dispersion $\alpha\in (0,1)$, and the non-local nature of the operator $\partial_x(1+D^{\alpha})^{-1}$, which is associated to solutions of the linear equation $\partial_t u+\partial_x u+D^{\alpha}\partial_t u=0$. Following the strategies used for the fKdV in weighted spaces (see \cite{Ri,GFFLGP1}), we will use several fractional derivatives (see \eqref{d1b} below) and commutator estimates. Here, we remark the commutator estimate of Lemma \ref{commA}, which is useful to handle the operator $\partial_x(1+D^{\alpha})^{-1}$.

In the next theorem, we deduce some UCP for the Cauchy problem \eqref{fBBM}.

\begin{theorem}\label{Main3} Let $\alpha\in (0,1)$, $r=(\frac{3}{2}+\alpha\big)^{-}$ (i.e., $r=\frac{3}{2}+\alpha-\epsilon$, $0<\epsilon\ll 1$), $s>s_\alpha$  where $s_{\alpha}=1$ if $\alpha\in[\frac{1}{2},1)$, and $s_{\alpha}=2-2\alpha$, if $\alpha\in (0,\frac{1}{2})$. Let $u\in C([0,T];Z_{s,r})$ be a  solution of the initial value problem \eqref{fBBM}, such that  if at times $0\leq t_1<t_2<T$, it holds that $u(t_j)\in Z_{s,\frac32+\alpha}, j=1,2$, then the following identity holds true:
\begin{equation*}
  \hat{u}(0,t_1)=\int_{\mathbb{R}} u(x,t_1)\, dx=-\frac{1}{(t_2-t_1)}\int_{t_1}^{t_2}\int_{\mathbb{R}}(u(x,\tau))^2\, dx d\tau.
\end{equation*}
In particular, $\int u(x,t_1)\, dx \geq 0$ if and only if $u$ is identically zero on $[t_1,\infty)$.
\end{theorem}

In the proof of the previous theorem we show the last equivalence with $u\equiv 0$ in $[t_1,t_2]$, however due to the uniqueness of local solutions of \eqref{fBBM} and the fact that the null solution is global, in the statement of Theorem \ref{Main3}, we have used that $u\equiv 0$ in $[t_1,t_2]$ if and only if $u\equiv 0$ in $[t_1,\infty)$.

Theorem \ref{Main3} establishes the maximum admissible polynomial decay rate for solutions of \eqref{fBBM} in $L^2$ weighted Sobolev spaces. To see this, consider a sufficiently regular initial data $\varphi \in L^{2}(|x|^{2(\frac{3}{2}+\alpha)}\, dx)\setminus\{0\}$ (e.g., $\varphi$ in the Schwartz class of functions) such that $\int \varphi(x)\, dx\geq 0$. Then by Theorem \ref{Main2}, there exists a time $T>0$ and a unique solution of \eqref{fBBM} such that $u\in C([0,T];Z_{s,(\frac{3}{2}+\alpha)^{-}})$, $s>s_{\alpha}$. However, Theorem \ref{Main3} implies that for all $T_1>0$,
\begin{equation*}
 u\notin L^{\infty}([0,T_1];L^{2}(|x|^{2(\frac{3}{2}+\alpha)}\, dx)). 
\end{equation*}
Consequently, our results of Theorem \ref{Main2} are sharp in the previous sense.

We also note that in the case of fKdV with dispersion $\alpha\in (-1,2)$, $\alpha\neq 0$ (see,  \cite{Ri,GFFLGP1,FoPo}), it is known that the maximum polynomial decay of solutions with arbitrary initial data in $Z_{s,r}$-spaces is $r=\frac{3}{2}+\alpha$, and assuming condition $\int \varphi(x)\, dx=0$ for the initial data, this extends to $r=\frac{5}{2}+\alpha$. Contrary to this, $\frac{3}{2}+\alpha$ is the maximum decay for solutions of fBBM. Thus, we observe that equation fBBM is more restrictive in weighted spaces than fKdV. On the other hand, the maximum decay for fKdV is obtained by deducing UCP at three different times (or two times and an extra condition). However, fBBM requires two-time conditions to obtain the analog UCP.  Finally, we note that the results obtained for fBBM with $\alpha\in (0,1)$ are compatible with those for the case $\alpha=1$ in \cite{FoRoSa}.

\begin{remark} (a) Assuming $\alpha\in (\frac{1}{2},1)$, our proof of Theorems \ref{Main2} and \ref{Main3} extend to any regularity $s>\frac{3}{2}-\alpha$, if one can guarantee local existence of solutions for the IVP \eqref{fBBM} in $H^{(\frac{3}{2}-\alpha)^{+}}(\mathbb{R})$. See Remark \ref{Remarkextension} below.

(b) The results of Theorems \ref{Main2} and \ref{Main3}, are easily extended to dispersions $\alpha>0$ with $\alpha \notin 2\mathbb{Z}$ in \eqref{fBBM}. For instance, we can use the standard LWP results for \eqref{fBBM} in $H^s(\mathbb{R})$, $s>\frac{3}{2}$, and our energy estimates and arguments developed below. Thus, for any range of dispersion $\alpha>0$ with $\alpha \notin 2\mathbb{Z}$, we obtain that $r=\frac{3}{2}+\alpha$ is the decay limit for solutions of fBBM in spaces $Z_{s,r}$.

(c) Given the generality of our estimates, our results also extend to the following generalized fBBM equation    
\begin{equation*}
\partial_t u+\partial_x u+ D^{\alpha}\partial_{t} u+\partial_x (u^k) = 0, \qquad t, x\in \R,    
\end{equation*}
where $k\in \mathbb{Z}^{+}$, $k\geq 3$, and $\alpha\in (0,1)$. However, in this case, the identity obtained in Theorem \ref{Main3} turns out to be
\begin{equation*}
  \int_{\mathbb{R}} u(x,t_1)\, dx=-\frac{1}{(t_2-t_1)}\int_{t_1}^{t_2}\int_{\mathbb{R}}(u(x,\tau))^k\, dx d\tau.
\end{equation*}
Hence, when $k$ is an even number, one has that  $\int u(x,t_1)\, dx \geq 0$ if and only if $u$ is identically zero on $[t_1,\infty)$. In particular, when $k$ is an even integer, the maximum polynomial decay rate for arbitrary initial condition is $\frac{3}{2}+\alpha$. However, we do not know if such a decay rate bound still holds with $k$ being an odd number. 
\end{remark}

This manuscript is organized as follows. In Section \ref{SectPrelim}, we detail preliminaries and estimates, which are fundamental to derive our main results, among which are commutator and fractional derivative estimates. In this part, we also derive Lemma \ref{commA}, which is a key result to obtain commutator estimates for the operator $A$ (see \eqref{Aoperator} below). In Section \ref{SectionLWP}, we prove the persistence result in weighted spaces detailed in Theorem \ref{Main2}. We continue with Section \ref{SectUnique}, in which we prove our unique continuation principles for solutions of \eqref{fBBM} at two different times, i.e., we deduce Theorem \ref{Main3}.


\subsection*{Notation} 
\begin{itemize}
\item $a\lesssim b$ means that there exists $C>0$ such that $a\leq Cb$. Similarly, we set $a\gtrsim b$. We will say that $a\sim b$, whenever $a\lesssim b$ and $b\lesssim a$. 
    \item $[B,C]=BC-CB$ denotes the commutator between operators $B$ and $C$.
    \item $\langle x\rangle=(1+|x|^2)^{\frac{1}{2}}$, $x\in \mathbb{R}$.
    \item   $\|\cdot\|_{L^p}$ denote the usual norm of the Lebesgue space $L^p(\mathbb{R})$, $1\leq p \leq \infty$. Since several of your estimates focus on the space $L^2(\mathbb{R})$, we denote its norm as $\|\cdot\|_2=\|\cdot\|_{L^2}$. 
    \item $\wedge$ denotes the Fourier transform, $\vee$ denotes the inverse Fourier transform.
    \item Given $s\in \mathbb{R}$, $D^sf=(|\xi|^s\widehat{f}(\xi))^{\vee}$, and $J^sf=(1-\partial_x^2)^{\frac{s}{2}}f=\big(\langle \xi \rangle^s\widehat{f}(\xi)\big)^{\vee}$.
    \item In all paper, 
    \begin{equation}\label{Aoperator}
A =-\partial_x(1+D^\alpha)^{-1}.    
\end{equation}
\item $H^s(\mathbb{R})$
 denotes the Sobolev space of order $s$.
 \item $\|\cdot\|_{s,2}=\|J^s( \cdot)\|_{L^2}$ denotes the $H^s$-norm.
 \item $Z_{s,r}=H^s(\mathbb{R})\cap L^2(|x|^{2r}\,dx).$ 

 \item  $\{e^{tA}\}$ is the one-parameter unitary operator generated by $A$. 

\item Since $\{e^{tA}\}$ is the linear group of solutions of the equation $\partial_t u-Au=0$. It follows that the integral formulation of \eqref{fBBM} is 
\begin{equation}\label{IE}
 u(t)=e^{tA}\varphi+\int_0^t e^{(t-\tau)A}A( u^2)(\tau)\, d\tau.
\end{equation}
 
 \end{itemize}


\section{Preliminary results}\label{SectPrelim}

In order to obtain group estimates, next lemma provides some formulae for derivatives of the unitary group $\{e^{tA}\}$ associated to the fBBM equation in Fourier space, which easily follows from a direct computation.  

\begin{proposition}\label{derivatives} 
Let $F(t,\xi)= e^{-ia(\xi)t}$, where $a(\xi)=\dfrac{\xi}{1+|\xi|^\alpha}.$ Then,
\begin{align}&\partial_{\xi}F(t,\xi)  =-it\frac{1+(1-\alpha)|\xi|^\alpha}{(1+|\xi|^\alpha)^{2}}F(t,\xi),  \label{der1}\\
&\partial_{\xi}^{2}F(t,\xi) =(-it)^2 \frac{(1+(1-\alpha)|\xi|^\alpha)^2}{(1+|\xi|^\alpha)^{4}}F(t,\xi) \nonumber\\ 
&+it\frac{\alpha(\alpha+1)|\xi|^{\alpha-1}\sgn(\xi)}{(1+|\xi|^\alpha)^{3}}F(t,\xi)+it\frac{\alpha(1-\alpha)|\xi|^{2\alpha-1}\sgn(\xi)}{(1+|\xi|^\alpha)^{3}}F(t,\xi).  \label{der2}
\end{align}
\end{proposition}

We will use the following fractional Leibniz rules established in \cite{KatPoc}, \cite[Theorem 1]{GrafakosOh2014}, \cite{BoLi} for the $L^\infty$ endpoint inequality, and \cite{OhWu2020}  for the $L^1$ endpoint inequality with $p=1=p_2=p_4$.  We also refer to \cite{KenigPonceVega1993}.

\begin{lemma}\label{leibnizhomog}
Let $p\in[1,\infty]$, $s\geq 0$ with $s$ not an odd integer in the case $p = 1$. Then
\begin{equation}
    \|D^s(fg)\|_{L^p} \lesssim \|f\|_{L^{p_1}}\|D^sg\|_{L^{p_2}} + \|g\|_{L^{p_3}}\|D^sf\|_{L^{p_4}}
\end{equation}
and
\begin{equation}
    \|J^s(fg)\|_{L^p} \lesssim \|f\|_{L^{p_1}}\|J^sg\|_{L^{p_2}} + \|g\|_{L^{p_3}}\|J^sf\|_{L^{p_4}}
\end{equation}
with $ \frac{1}{p} = \frac{1}{p_1} + \frac{1}{p_2} = \frac{1}{p_3} + \frac{1}{p_4}$.
\end{lemma}

\subsection{Commutator estimates} Next, we introduce some key commutator estimates. 

\begin{lemma}\label{dmp1} Let $\hil$ denote
the Hilbert transform. Then for any $p\in (1,\infty)$ and any $l, m
\in \Z^{+}\cup\{0\}$ there exists $c = c(p; l; m) > 0$ such that
\begin{equation}\label{c-dmp1}
\|\partial_x^l [\hil;\psi]\partial_x^m f\|_{L^p}\le
c\,\|\partial_x^{m+l}\psi\|_{L^{\infty}}\|f\|_{L^p}.
\end{equation}
\end{lemma}
The case $l+m=1$ in Lemma \ref{dmp1} is the classical Calderon's commutator estimate, \cite{Ca}, for the general case see Lemma 3.1 in \cite{DaMcPo}. We will also use the following commutator estimate deduced in \cite{Ponce,KatPoc} and Theorem 1.9 in \cite{Li}.

\begin{lemma}\label{dmp2}
Let $\alpha \in [0, 1)$, $\beta\in(0, 1)$ with $\alpha+\beta \in [0,
1]$. Then for any $p \in  (1,\infty)$ there exists $c = c(\alpha;\beta; p) > 0$ such that
\begin{equation}\label{OR2}
\|D^{\alpha} [D^{\beta}; \psi]D^{1-(\alpha+\beta)} f\|_{L^p}\le
c\,\|\partial_x \psi\|_{L^{\infty}} \|f\|_{L^p},
\end{equation}
\end{lemma}
See the proof of Lemma \ref{dmp2}  in \cite[Proposition 3.10]{Li} (see also Proposition 3.2 in \cite{DaMcPo}). 

Using \eqref{OR2}, we can derive the following fundamental commutator estimate for the operator $A$.
\begin{lemma}\label{commA}
Let $\alpha \in (0, 1]$, $1<p<\infty$, $f\in L^p(\mathbb{R})$ and let $g$ be a measurable function such that $\partial_x g\in L^{\infty}(\mathbb{R})$. Then we have the nice commutator estimate for the operator $A=-\partial_x (1+D^{\alpha})^{-1}$ 
\begin{equation}\label{comA}
\|[A,g]f\|_{L^p}\lesssim \|\partial_x g\|_{L^\infty} \|f\|_{L^p}.
\end{equation}
\end{lemma}

\begin{proof}
By opening the commutator and factorizing the operator $(1+D^{\alpha})^{-1}$, we distribute the derivative $\partial_x$ to obtain
\begin{equation}\label{eq1}
\begin{aligned}
\big[A,g \big]f
=&A(gf)-gA f\\
=&-(1+D^{\alpha})^{-1}\Big(\partial_x gf+g\partial_x f+g Af+D^{\alpha}(g Af) \Big).
\end{aligned}    
\end{equation}
Writing
\begin{equation*}
\begin{aligned}
D^{\alpha}\big(g Af\big)=[D^{\alpha},g] A f+g D^{\alpha} A f,   
\end{aligned}    
\end{equation*}
and plugging it in the previous identity  \eqref{eq1}, we arrive at
\begin{equation}\label{eq2}
\begin{aligned}
\big[A, & g \big]f
\\
=&-(1+D^{\alpha})^{-1}\Big(\partial_x gf+g\partial_x f+g (1+D^{\alpha}) Af+[D^{\alpha},g] Af\Big)\\
=&-(1+D^{\alpha})^{-1}\Big(\partial_x gf+[D^{\alpha},g] A f\Big),
\end{aligned}    
\end{equation}
where we have used that $(1+D^{\alpha}) A=-\partial_x$. We denote by $G_{\alpha}\in L^{1}(\mathbb{R})$ the kernel representation of the operator $(1+D^{\alpha})^{-1}$. For further properties of $G_{\alpha}$, we refer to \cite[Lemma C.1]{FraLenzSilv2016}.

Consequently, using \eqref{eq2} and that $G_{\alpha}\in L^{1}(\mathbb{R})$, an application of Young's inequality yields 
\begin{equation}
\begin{aligned}
\|\big[A,g \big]f\|_{L^p}\lesssim & \|G_{\alpha}\ast \Big(\partial_x gf+[D^{\alpha},g] A f\Big) \|_{L^p}\\
\lesssim & \|\partial_x g\|_{L^{\infty}}\|f\|_{L^p}+ \|[D^{\alpha},g] A f\|_{L^p}. 
\end{aligned}    
\end{equation}
To estimate the second term on the right-hand side of the above inequality, we use the decomposition $\partial_x=-\hil D=-\hil D^{1-\alpha}D^{\alpha}$ to apply Lemma \ref{dmp2} to get
\begin{equation*}
\begin{aligned}
   \|[D^{\alpha},g] A f\|_{L^p} = &   \|[D^{\alpha},g]D^{1-\alpha} \hil D^{\alpha}(1+D^{\alpha})^{-1}f\|_{L^p}\\
   \lesssim &  \|\partial_x{g}\|_{L^{\infty}}\| \hil D^{\alpha}(1+D^{\alpha})^{-1}f\|_{L^p}\\
   \lesssim &  \|\partial_x{g}\|_{L^{\infty}}\| f\|_{L^p},
\end{aligned}    
\end{equation*}    
where the above estimate follows from writing $ D^{\alpha}(1+D^{\alpha})^{-1}f=f-(1+D^{\alpha})^{-1}f=f-G_{\alpha}\ast f$, and using that the Hilbert transform establishes a bounded operator in $L^p(\mathbb{R})$. Gathering the previous results complete the proof.

\end{proof}


\subsection{Fractional derivative and interpolation estimates}

We recall the following  characterization of the $L^p_s(\R^n)\!=\!(1-\Delta)^{-\frac{s}{2}}L^p(\R^n)$ spaces given  in \cite{St1}, (see \cite{AroSmit} for the case $p=2$).

 \begin{theorem}[\cite{St1}]
 \label{theorem9b}
Let $b\in (0,1)$ and $\; \frac{ 2n}{n+2b}< p< \infty$. Then $f\in  L^p_b(\R^n)$ if and only if
\begin{equation}\label{d1}
\;(a)\hskip10pt f\in L^p(\R^n),\hskip3cm\text{\hskip5cm }\\
\end{equation}
\begin{equation}\label{d1b}
\;(b)\hskip10pt \mathcal D^b f(x)=\Big(\int_{\R^n}\frac{|f(x)-f(y)|^2}{|x-y|^{n+2b   }}dy\Big)^{\frac{1 }{2}}\in L^p(\R^n),
\text{\hskip2cm}
\end{equation}
with
 \begin{equation}
\label{d1-norm}
\|f\|_{b,p}\equiv  \|(1-\Delta)^{b} f\|_p=\|J^sf\|_p\simeq \|f\|_p+\|D^b    f\|_p\simeq \|f\|_p
+\|\mathcal D^b  f\|_p.
\end{equation}
 \end{theorem}
It follows from \eqref{d1b} that for $p=2$ and $b\in(0,1)$ one has
\begin{equation}\label{pointwise2}
\|\mathcal D^b (fg)\|_2\leq \|f\, \mathcal D^b g\|_2  +\|g\, \mathcal D^b f\|_2,
\end{equation}
and
\begin{equation}\label{pointwise2a}
\|\mathcal D^b f\|_2= c\, \|D^b f\|_2.
\end{equation}
The definition \eqref{d1b} also yields
\begin{equation}\label{Linftybound}
  \|\mathcal{D}^{b}f\|_{L^{\infty}}\lesssim \|f\|_{L^{\infty}}+\|\partial_x f\|_{L^{\infty}}.  
\end{equation}
We also need the following interpolation inequality whose proof is a consequence of Theorem \ref{theorem9b} and complex interpolation
\begin{lemma}[\cite{FoPo}]\label{lemma1}
Let $s,\,b>0$. Assume that $ J^{s}f=(1-\Delta)^{\frac{s}{2}}f\in L^2(\mathbb R)$ and 
$\ji x\jd^{b}f=(1+|x|^2)^{\frac{b}{2}}f\in L^2(\mathbb R)$. Then for any $\theta \in (0,1)$
\begin{equation}\label{complex}
\|J^{ \theta s}(\ji x\jd^{(1-\theta) b} f)\|_2\leq c\|\ji x\jd^b f\|_2^{1-\theta}\,\|J^{s}f\|_2^{\theta}.
\end{equation}
More generally, let $r_1$, $r_2\geq 0$, and $w_{N}$ be as in \eqref{trun}. If $J^s f, \ji x\jd^{r_1b}w_{N}^{r_2b}f \in L^2(\mathbb{R})$, it follows 
\begin{equation}\label{complex2}
\|J^{ \theta s}\big((\ji x\jd^{r_1} w_{N}^{r_2})^{(1-\theta) b} f\big)\|_2\leq c\|(\ji x\jd^{r_1}w_{N}^{r_2})^b f\|_2^{1-\theta}\,\|J^{s}f\|_2^{\theta}
\end{equation}
with a constant $c$ independent of $N$.
\end{lemma}

\begin{proof}
The interpolation inequality \eqref{complex} was deduced in \cite[Lemma 1]{FoPo}. Now, setting $\rho=\ji x\jd^{r_1} w_{N}^{r_2}$, we have that $|\rho'/\rho|+|\rho''/\rho|\lesssim 1$ with implicit constant independent of $N$. This in turn implies that  \eqref{complex2} is a consequence of the proof of Lemma 1 in \cite{FoPo}.   
\end{proof}

\begin{remark}\label{remarkinterp} 
Using Plancherel's identity and \eqref{complex} in the frequency domain, a similar interpolation result to \eqref{complex} is valid for $\ji x\jd^{(1-\theta) b} J^{ \theta s}f$.     
\end{remark}

Next result will be useful in our arguments. Its proof directly follows from the definition of the fractional derivative $\mathcal{D}^{b}$ in \eqref{d1b}, and it is given in \cite{GFFLGP1}.
 \begin{proposition}\label{prop3}
For any $\theta\in (0,1)$, $\alpha>0$ and $\psi\in C_0^{\infty}(\R)$ with $0\leq\psi\leq1$ and $\psi(\xi)=1$ for $|\xi|\leq 1$
\begin{equation}\label{steinderiv}
\mathcal{D}^{\theta}\Big(|\xi|^{\alpha}\psi(\xi)\Big)(\eta)\sim
\begin{cases}
c\,|\eta|^{\alpha-\theta}+c_1,\hskip15pt\alpha\neq \theta,\hskip10pt |\eta| \ll 1,\\
\\
c\,(-\ln|\eta|)^{\frac{1}{2}},\hskip20pt \alpha=\theta,\hskip10pt|\eta| \ll 1,\\
\\
\dfrac{c}{|\eta|^{1/2+\theta}},\hskip78pt |\eta|\gg1,
\end{cases}
\end{equation}
with $\mathcal{D}^{\theta}\Big(|\xi|^{\alpha}\psi(\xi)\Big)(\cdot) $ continuous in $\eta\in\R-\{0\}$.
In particular, one has that
\begin{equation*}
\mathcal{D}^{\theta}\Big(|\xi|^{\alpha}\psi(\xi)\Big)\in L^2(\R)
\text{\hskip10pt if and only if\hskip10pt}\theta<\alpha+1/2.
\end{equation*}
Similar result holds for $\mathcal{D}^{\theta}\Big(|\xi|^{\alpha}\sgn(\xi)\psi(\xi)\Big)(\eta)$.
\end{proposition}

\begin{proposition}\label{prop4}
For any $\theta\in (0,1)$, $0<\beta<1/2$, and $\psi$ a compactly supported smooth cut-off function
\begin{equation}\label{negsteinderiv}
\mathcal{D}^{\theta}\Big(|\xi|^{-{\beta}}\psi(\xi)\Big)(\eta)\lesssim
(\|\psi\|_{L^\infty}+\|\langle \xi \rangle\psi\|_{L^\infty})|\eta|^{-\beta-\theta}
\end{equation}
for all  $\eta\neq 0$. 
\end{proposition}
For the proof of \eqref{negsteinderiv}, we refer to \cite{Ri}.  Next, we deduce the following fractional derivative estimate for the function $\frac{1}{1+|\xi|^{\alpha}}$.

\begin{proposition}\label{prop5}
Let $\alpha>0$. For any $\theta\in (0,1)$ and $\psi$ a compactly supported smooth cut-off function
\begin{equation}\label{negsteinderiv2}
\mathcal{D}^{\theta}\Big(\frac{1}{1+|\xi|^{\alpha}}\psi(\xi)\Big)(\eta)\lesssim \mathcal{D}^{\theta}(\psi(\xi))(\eta)+\mathcal{D}^{\theta}(|\xi|^{\alpha}\psi(\xi))(\eta).
\end{equation}
for all  $\eta\in \mathbb{R}$.
\end{proposition}

\begin{proof}
Setting $G(\xi):=\frac{\psi(\xi)}{1+|\xi|^{\alpha}}$ and grouping factors, it is seen that\begin{equation*}
\begin{aligned}
\big|G(\xi)-G(\eta)\big|=& \frac{|(1+|\eta|^{\alpha})\psi(\xi)-(1+|\xi|^{\alpha})\psi(\eta)|}{(1+|\xi|^{\alpha})(1+|\eta|^{\alpha})} \\
\leq & \frac{(1+|\xi|^{\alpha}+|\eta|^{\alpha})|\psi(\xi)-\psi(\eta)|}{(1+|\xi|^{\alpha})(1+|\eta|^{\alpha})}\\
&+\frac{||\xi|^{\alpha}\psi(\xi)-|\eta|^{\alpha}\psi(\eta)|}{(1+|\xi|^{\alpha})(1+|\eta|^{\alpha})}, 
\end{aligned}    
\end{equation*}
for any $\xi,\eta \in \mathbb{R}$. Hence, we get
\begin{equation*}
\big|G(\xi)-G(\eta)\big|\lesssim |\psi(\xi)-\psi(\eta)| +||\xi|^{\alpha}\psi(\xi)-|\eta|^{\alpha}\psi(\eta)|. 
\end{equation*}
By squaring the above identity, multiplying the resulting expression by $|\xi-\eta|^{-1-2\theta}$, and integrating over $\xi$ in $\mathbb{R}$, the desired result is a consequence of the definition of the fractional derivative in Theorem \ref{theorem9b}.
\end{proof}

We conclude this section with the following proposition which will be useful for controlling $(1+|\xi|^{\alpha})^{-j}$.

\begin{proposition}\label{negativeleibnizprop} Let $\alpha>0$, $0<\theta<1$, $j\geq 1$ be integer. Let $g$ be a measurable function with compact support, and $\psi$ be a compactly supported smooth cut-off function such that $g\psi=g$ almost everywhere. Assume $g, \mathcal{D}^{\theta}(g), g\mathcal{D}^{\theta}(|\xi|^{\alpha}\psi) \in L^2(\mathbb{R})$. Then it follows
\begin{equation}
\|\mathcal{D}^{\theta}\Big(\frac{g}{(1+|\xi|^{\alpha})^j}\Big)\|_2\lesssim  \|g\|_2+\|g\mathcal{D}^{\theta}(|\xi|^{\alpha}\psi)\|_2+\|\mathcal{D^{\theta}}(g)\|_2.
\end{equation}    
\end{proposition}

\begin{proof} We will only show the case $j=2$ as the others follow from our ideas below and using the type of Leibniz's rule \eqref{pointwise2}. Thus, two applications of \eqref{pointwise2} yield 
\begin{equation*}
\begin{aligned}
\|\mathcal{D}^{\theta}\Big(\frac{g}{(1+|\xi|^{\alpha})^2}\Big)\|_2\lesssim & \|\mathcal{D}^{\theta}\Big(\frac{\psi}{1+|\xi|^{\alpha}}\Big) \frac{g}{1+|\xi|^{\alpha}}\|_2\\
&+\|\frac{1}{1+|\xi|^{\alpha}}\|_{L^{\infty}}\|\mathcal{D}^{\theta}\Big(\frac{1}{1+|\xi|^{\alpha}}g\Big)\|_2 \\
\lesssim & \|\frac{1}{1+|\xi|^{\alpha}}\|_{L^{\infty}}\Big( \|\mathcal{D}^{\theta}\Big(\frac{\psi}{1+|\xi|^{\alpha}}\Big) g\|_2\\
& +\|\frac{1}{1+|\xi|^{\alpha}}\|_{L^{\infty}}\|\mathcal{D^{\theta}}(g)\|_2 \Big).
\end{aligned}    
\end{equation*}
Consequently, we use Proposition \ref{prop5} to get
\begin{equation*}
\begin{aligned}
\|\mathcal{D}^{\theta}\Big(\frac{g}{(1+|\xi|^{\alpha})^2}\Big)\|_2
\lesssim & \|\mathcal{D}^{\theta}\psi\|_{L^{\infty}} \|g\|_2+\|g\mathcal{D}^{\theta}(|\xi|^{\alpha}\psi)\|_2+\|\mathcal{D^{\theta}}(g)\|_2.
\end{aligned}    
\end{equation*}
    
\end{proof}


\section{Proof of Theorem \ref{Main2}}\label{SectionLWP}

We first recall the index $s_{\alpha}>0$, which is defined as $s_{\alpha}=1$,  if $\alpha\in [\frac{1}{2},1)$, and $s_{\alpha}=2-2\alpha$, if $\alpha\in (0,\frac{1}{2})$. In order to simplify the exposition of our arguments, we will divide the proof of Theorem \ref{Main2} into the following propositions:

\begin{proposition}\label{propTH1} Let $\alpha\in (0,1)$, $s>s_{\alpha}$, $r\in (0,1]$. Consider $\varphi \in Z_{s,r}$. Let $T>0$ and $u\in C([0,T];H^s(\mathbb{R}))$ be the unique solution of the Cauchy problem \eqref{fBBM} with initial condition $\varphi$ provided by Theorem \ref{theorem1}. Then, it follows
\begin{equation}
u\in C([0,T];H^{s}(\R)\cap L^2(|x|^{2r}\, dx)\equiv
Z_{s,r}).
\end{equation}
\end{proposition}

\begin{proposition}\label{propTH2} Let $\alpha\in (0,1)$, $s>s_{\alpha}$, $r \in(1,\frac{3}{2}+\alpha)$ if $\alpha\in (0,\frac{1}{2}]$, $r\in(1,2]$ if $\alpha\in (\frac{1}{2},1)$. Consider $\varphi \in Z_{s,r}$. Let $T>0$ and $u\in C([0,T];H^s(\mathbb{R}))$ be the unique solution of the Cauchy problem \eqref{fBBM} with initial condition $\varphi$ provided by Theorem \ref{theorem1}. Then, it follows
\begin{equation}
u\in C([0,T];Z_{s,r}).
\end{equation}
\end{proposition}

\begin{proposition}\label{propTH3} Let $\alpha \in (\frac{1}{2},1)$, $s>s_{\alpha}$, and  $r\in(2,\frac{3}{2}+\alpha)$. Consider $\varphi \in Z_{s,r}$. Let $T>0$ and $u\in C([0,T];H^s(\mathbb{R}))$ be the unique solution of the Cauchy problem \eqref{fBBM} with initial condition $\varphi$ provided by Theorem \ref{theorem1}. Then, it follows
\begin{equation}
u\in C([0,T];Z_{s,r}).
\end{equation}
\end{proposition}

As stated in the above propositions, the existence and uniqueness of local solutions for the IVP \eqref{fBBM} in $H^s(\mathbb{R})$, $s>s_{\alpha}$ is granted by Theorem \ref{theorem1}. Thus, the preceding propositions focus on showing that solutions of \eqref{fBBM} persist in space $L^2(|x|^{2r}\, dx)$.  Moreover, we note that regardless of the magnitude of the dispersion parameter $\alpha\in (0,1)$, solutions of the Cauchy problem \eqref{fBBM} always propagate weights $r$ within $(0,1]$. Such a result is described in Proposition \ref{propTH1}. However, the results of Propositions \ref{propTH2} and \ref{propTH3} establish that solutions of fBBM propagate fractional polynomial weights strictly less than $\frac{3}{2}+\alpha$. Therefore, we have that Propositions \ref{propTH1} and \ref{propTH2} prove Theorem \ref{Main2} for dispersions $\alpha\in (0,\frac{1}{2}]$; and at the same time Propositions \ref{propTH1}, \ref{propTH2}, and \ref{propTH3}  prove Theorem \ref{Main2} for the case $\alpha\in (\frac{1}{2},1)$.

\begin{remark}
 Due to the similarity of our arguments, we will only prove Propositions \ref{propTH1} and \ref{propTH3}. The deduction of Proposition \ref{propTH2} follows from similar arguments.
\end{remark}

\begin{proof}[Proof of Proposition \ref{propTH1}]

We use the differential equation and the local theory established in \cite{LiPiSa} such that
$u\in C([0,T]:H^{s}(\R))$ with $s>s_{\alpha}$ exists on the time interval $[0,T]$ and it is the limit of smooth solutions.  As a consequence, we note that equation in \eqref{fBBM} is equivalent to 
\begin{equation*}
  \partial_t u -Au-A(u^2)=0. 
\end{equation*}
Thus, in the following, when we mention energy estimates for equation \eqref{fBBM}, we refer to the equivalent equation above. 

Next, we define for $\theta\in(0,1]$
\begin{equation}\label{trun}
w_{N}^{\theta}=
\begin{cases}
\ji x\jd^{\theta}=(1+x^2)^{\frac{\theta}{2}}, \text{\hskip1cm if}\;\; |x|\le N,\\
(2N)^{\theta}, \text{\hskip85pt if}\;\; |x|\ge 3N,
\end{cases}
\end{equation}
with $w^{\theta}_N$ smooth, even, nondecreasing for $x\ge 0$.

Here we set $r=\theta\in [0,1]$. We perform the usual energy estimates by multiplying the equation in \eqref{fBBM} by $(\wnt)^2\,u$ and integrating it into the
$x$-variable in order to obtain
\begin{equation*}\label{e1.2}
\frac12 \frac{d}{dt}\int (\wnt\,u)^2\,dx-
\underset{\ma}{\underbrace{\int \
\wnt\,Au\,\wnt\,u\,dx}}-
\underset{\mb}{\underbrace{\int \wnt A(u^2)\,\wnt
u\,dx}}=0.
\end{equation*}
Regarding the integral $\ma$, we introduce a commutator and get
\begin{equation*}\label{e1.3}
 \ma=\int [\wnt;A]\,u\,\wnt u\,dx+\int A\wnt
u\,\wnt u\,dx,
\end{equation*}
and we notice that the second integral on the right hand side vanishes since the operator $A$ is skew-symmetric, and hence from Cauchy-Schwarz inequality and Lemma \ref{commA}, we get
\begin{equation*}\label{e1.4}
 |\ma|\lesssim \|[\wnt;A]\,u\|_2\|\wnt u\|_2\lesssim \|u\|_2\|\wnt u\|_2,
\end{equation*}
where we have used that $|\partial_x(\wnt)|\lesssim 1$ with implicit constant independent of $N$. On the other hand, for the integral $\mb$ we have in much similar way that
\begin{equation*}\label{e1.5}
 \mb=\int [\wnt;A]\,u^2\,\wnt u\,dx+\int A\wnt
u^2\,\,\wnt u\,dx=\mb_1+\mb_2.
\end{equation*}
Therefore, since $s> s_{\alpha}>1/2$,  we now obtain that
\begin{equation*}\label{e1.6}
 |\mb_1|\lesssim \|[\wnt;A]\,u^2\|_2\|\wnt u\|_2\lesssim \|u^2\|_{s,2}\|\wnt u\|_2\lesssim \|u\|^2_{s,2}\|\wnt u\|_2.
\end{equation*}
We remark these {\it a priori} estimates are valid for $\theta \in (0,1]$.

Regarding the integral $\mb_2$ we apply Cauchy-Schwarz inequality to obtain
\begin{equation*}\label{e1.7}
|\mb_2|\lesssim \|A\wnt u^2\|_2\|\wnt u\|_2,
\end{equation*}
and hence it remains to estimate the first norm in this last inequality. Now, the argument relies on applying Leibniz rule and appropriate Holder, Sobolev inequalities provided $\alpha \in (0,1)$, and the interpolation Lemma \ref{lemma1}. Indeed, an application of Lemma \ref{leibnizhomog} yields
\begin{equation*}\label{e1.8}
\begin{split}
\|A\wnt u^2\|_2&\lesssim \|J^{1-\alpha}(w_N^{\frac{\theta}{2}}\,u)^2\|_2\\
&\lesssim \|w_N^{\frac{\theta}{2}}\,u\|_{L^p}\|J^{1-\alpha}(w_N^{\frac{\theta}{2}}\,u)\|_{L^q}\\
&=:\mb_{2,2},
\end{split}
\end{equation*}
where $p, q\in (1,\infty]$ are such that $\frac{1}{2}=\frac{1}{p}+\frac{1}{q}$. In what follows, depending on the values of $\alpha\in (0,1)$, we will choose $p$, $q$ suitable for bounding the above inequality.


$i)$ \underline {\bf Case $\alpha \in[\frac{1}{2},1)$:}

Within this range of $\alpha \in (\frac{1}{2},1)$, we choose $p=\frac{4}{2\alpha-1}$ and $q=\frac{4}{3-2\alpha}$ in $\mb_{2,2}$. Thus, we apply  Sobolev inequality with $\frac1p=\frac12-\frac{3-2\alpha}{4},$ and the interpolation Lemma \ref{lemma1} so that
\begin{equation*}\label{e1.11}
\begin{aligned}
  \|w_N^{\frac{\theta}{2}}\,u\|_{L^p}\lesssim \|J^{\frac{3-2\alpha}{2}\cdot\frac12}\big(w_N^{\theta\cdot\frac12}\,u\big)\|_{2}\lesssim & \|J^{\frac{3}{2}-\alpha}\,u\|^{\frac{1}{2}}_2\|\wnt\,u\|^{\frac{1}{2}}_{2}\\
\lesssim &  \|u\|^{\frac{1}{2}}_{s,2}\|\wnt\,u\|^{\frac{1}{2}}_{2}.   
\end{aligned}
\end{equation*}
By applying Sobolev inequality again, with $\frac1q=\frac12-\frac{2\alpha-1}{4}$ so that 

\begin{equation*}\label{e1.12}
\begin{aligned}
\|J^{1-\alpha}(w_N^{\frac{\theta}{2}}\,u)\|_{L^q}\lesssim \|J^{\frac{2\alpha-1}{4}+1-\alpha}(w_N^{\frac{\theta}{2}}\,u)\|_{2}\lesssim & \|J^{\frac{3-2\alpha}{2}\cdot\frac12}(w_N^{\theta\cdot\frac12}\,u)\|_{2}\\
\lesssim & \|u\|^{\frac{1}{2}}_{s,2}\|\wnt\,u\|^{\frac{1}{2}}_{2}.    
\end{aligned}
\end{equation*}
In the case $\alpha=\frac{1}{2}$, the natural choice is $p=\infty$, $q=2$. Then applying Sobolev embedding $H^{\frac{1}{2}^{+}}(\mathbb{R})\hookrightarrow L^{\infty}(\mathbb{R})$, and using that the regularity in this case satisfies $s>1$, the argument for estimating $\mb_{2,2}$ is similar to the one given above, so we omit its deduction. In summary, we have for the present case  
\begin{equation*}\label{e1.13}
|\mb_2|\lesssim \|u\|_{s,2}\|\wnt\,u\|^2_{2}
\end{equation*}
and 
\begin{equation*}\label{e1.14}
|\mb|\lesssim \|u\|_{s,2}\|\wnt\,u\|_{2}(\|u\|_{s,2}+\|\wnt\,u\|_{2}).
\end{equation*}

\begin{remark}\label{Remarkextension}
We note that the above arguments are valid for regularity $s>\frac{3}{2}-\alpha$, thus our proof of Theorem \ref{Main2} extends to $H^s(\mathbb{R})$ with $s>\frac{3}{2}-\alpha$, whenever $\alpha\in [\frac{1}{2},1)$.    
\end{remark}

$ii)$ \underline {\bf Case $0<\alpha < \frac{1}{2})$:}

In this case, we assume that the regularity of the initial data is taken to be $s>s_{\alpha}=2-2\alpha$ and proceed to complete the energy estimates for $\mb_{2,2}$ with the choice we have made for $\alpha=\frac{1}{2}$, that is, we consider $p=\infty$ and $q=2$ and apply Sobolev inequality and the interpolation Lemma \ref{lemma1} to get

\begin{equation*}\label{e1.15}
\begin{split}
 \mb_{2,2}&\lesssim \|w_N^{\frac{\theta}{2}}\,u\|_{L^\infty}\|J^{1-\alpha}\big(w_N^{\frac{\theta}{2}}\,u\big)\|_{2}\\
          &\lesssim \|J^{\frac{1+\e}{2}}\big(w_N^{\frac{\theta}{2}}\,u\big)\|_{2}\|J^{2(1-\alpha)\cdot\frac12}\big(w_N^{\theta\cdot\frac12}\,u\big)\|_{2}\\
          &\lesssim \|J^{1+\e}u\|^{\frac{1}{2}}_{2}\|\wnt\,u\|^{\frac{1}{2}}_{2}\|J^{2(1-\alpha)} u\|^{\frac{1}{2}}_{2}\|\wnt\,u\|^{\frac{1}{2}}_{2}\\
          &\lesssim \|u\|_{s,2}\|\wnt\,u\|_{2},
\end{split}
 \end{equation*}

provided $\e>0$ is small enough so that $2-2\alpha>1+\e$. This is the same estimate we obtained for the former case, and hence we conclude that for every $\alpha \in (0,1)$ and $\theta\in (0,1]$, it holds the weighted energy estimate:

\begin{equation*}\label{e1.16}
\frac12 \frac{d}{dt}\int (\wnt\,u)^2\,dx\lesssim |\ma|+|\mb|
\lesssim \|u\|_{s,2}\|\wnt\,u\|_{2}(1+\|u\|_{s,2}+\|\wnt\,u\|_{2})
\end{equation*}
and therefore

\begin{equation*}\label{e1.17}
\frac{d}{dt}(\|\wnt\,u\|_{2})\lesssim \|u\|_{s,2}(1+\|u\|_{s,2}+\|\wnt\,u\|_{2}),
\end{equation*}
where $0< \theta\leq 1$, and the implicit constant is independent of $N$. Consequently, using Gronwall's inequality, and letting $N\to \infty$, we obtain
\begin{equation*}
  u\in L^{\infty}([0,T],L^2(|x|^{2\theta}\, dx)).  
\end{equation*}
To deduce $u\in C([0,T],L^2(|x|^{2\theta}\, dx))$, one considers the sequence $(\wnt u)_{n\in \mathbb{Z}^{+}}\subset C([0,T],L^2(\mathbb{R}))$ and reapply the above argument to find that it is a Cauchy sequence. This completes the proof of Proposition \ref{propTH1}.

\end{proof}


\begin{proof}[Proof of Proposition \ref{propTH3}]

We consider now the weights $r=2+\theta$, with $r=2+\theta<\frac32+\alpha$ or equivalently $\theta < \alpha-\frac12$ and $\alpha\in(\frac12,1)$. Then multiplying the equation in \eqref{fBBM} by $(x^2\wnt)^2u$ and integrating we get
\begin{equation*}\label{e1.33}
\frac12 \frac{d}{dt}\int (x^2\wnt\,u)^2\,dx-
\underset{\ma}{\underbrace{\int \
\wnt\,x^2Au\,x^2\wnt\,u\,dx}}-
\underset{\mb}{\underbrace{\int\wnt  x^2A(u^2)\,x^2\wnt
u\,dx}}=0.
\end{equation*}
It should be mentioned that given the validity of Proposition \ref{propTH2}, the above differential equation is justified by the fact that $\wnt$ is bounded and $u\in C([0,T]; Z_{s,2})$, i.e., $x^2 u\in C([0,T]; L^2(\mathbb{R}))$. 

Let us consider first the term $\ma$. We write 
\begin{equation*}\label{e1.34}
\begin{split}
 \ma&=\int \wnt [x^2;A]u\,x^2\wnt u\,dx+\int \wnt A(x^2u)\,x^2\wnt u\,dx\\
	&=2\int \wnt A_1(xu)\,x^2\wnt u\,dx+\alpha(\alpha+1)\int \wnt\hil D^{\alpha-1}A_3 (u)\,x^2\wnt u\,dx\\
	&+(1-\alpha)\alpha \int \wnt\hil D^{2\alpha-1}A_3(u)\,x^2\wnt u\,dx+\int  [\wnt;A](x^2u)\,x^2\wnt u\,dx\\
	&+\int A(\wnt x^2u)\,\wnt x^2u\,dx\\
    	&=\ma_1+\ma_2+\ma_3+\ma_4+\ma_5,
\end{split}
\end{equation*} 
where  $[x;A]u=A_1u=( a'(\xi) \hat{u})^\vee$, and the nonlocal operator $A_3$ has Fourier symbol $(1+|\xi|^\alpha)^{-3}$ (see notation in Proposition \ref{derivatives}). We also point out that the fifth integral $\ma_5$  vanishes again. Now, to estimate $\ma_1$, we need the following result:
\begin{claim}\label{ClaimA1}
Let $\alpha\in (0,1)$,  $0<\theta<\min\{1,\alpha+\frac{1}{2}\}$, and $f$ be a measurable function such that $\langle x\rangle f\in L^2(\mathbb{R})$. Then
\begin{equation}\label{conclestim1}
\begin{aligned}
 \|\langle x\rangle^{\theta} & A_1 f\|_2\lesssim \|\langle x\rangle f\|_{2}.    
\end{aligned}    
\end{equation}
\end{claim}

\begin{proof}[Proof of Claim \ref{ClaimA1}]
 Let $\psi$, $\widetilde{\psi}$ be two compactly
supported smooth cut-off functions with $\psi(\xi)=1$ whenever $|\xi|\leq 1$, and $\psi \widetilde{\psi}=\psi$. Since $A_1$ defines a bounded operator in $L^2(\mathbb{R})$, we use the Fourier transform to change fractional decay into fractional differentiation (Theorem \ref{theorem9b}) to obtain
\begin{equation*}\label{e1.22}
\begin{split}
 \|\langle x\rangle^{\theta} & A_1 f\|_2\\
 \lesssim & \|A_1 f\|_2+\| |x|^{\theta} A_1 f\|_2\\
\lesssim & \|f\|_2+\|D^{\theta}_{\xi}\big(\frac{1+(1-\alpha)|\xi|^\alpha}{(1+|\xi|^\alpha)^2}\psi \hat{f}\big)\|_2+\|D^{\theta}_{\xi}\big(\frac{1+(1-\alpha)|\xi|^\alpha}{(1+|\xi|^\alpha)^2}(1-\psi) \hat{f}\big)\|_2\\
\lesssim & \|f\|_2+\|\mathcal{D}_{\xi}^{\theta}\big(\frac{\widetilde{\psi}}{1+|\xi|^{\alpha}}\big)\frac{(1+(1-\alpha)|\xi|^\alpha)}{1+|\xi|^{\alpha}}\psi \hat{f}\|_2\\
& +\|\mathcal{D}_{\xi}^{\theta}\big(\frac{\widetilde{\psi}}{1+|\xi|^{\alpha}}\big)(1+(1-\alpha)|\xi|^\alpha)\psi \hat{f}\|_2\\
&+\|\mathcal{D}_{\xi}^{\theta}\big(1+(1-\alpha)|\xi|^\alpha\psi \big)\hat{f}\|_2 + \|\mathcal{D}_{\xi}^{\theta}\Big(\frac{1+(1-\alpha)|\xi|^\alpha}{(1+|\xi|^\alpha)^2}(1-\psi)\Big)\widehat{f}\|_2 \\
&+\|\mathcal{D}^{\theta}_{\xi} \hat{f}\|_2,
 \end{split}
\end{equation*}
where we have used several times properties \eqref{pointwise2}, that $\psi\widetilde{\psi}=\psi$ and the fact $(1+|\xi|^{\alpha})^{-1}\widetilde{\psi}\in L^{\infty}(\mathbb{R})$. Let us proceed with the estimate of the inequality above. Since $\theta<\alpha+\frac{1}{2}$, we use \eqref{negsteinderiv2} and Proposition \ref{prop3} to get
\begin{equation*}
\begin{aligned}
\|\mathcal{D}_{\xi}^{\theta}\big(\frac{\widetilde{\psi}}{1+|\xi|^{\alpha}}\big)\frac{(1+(1-\alpha)|\xi|^\alpha)}{1+|\xi|^{\alpha}}\psi \hat{f}\|_2\lesssim & \|\widehat{f}\|_2+\|\mathcal{D}_{\xi}^{\theta}(|\xi|^{\alpha}\widetilde{\psi})\psi \widehat{f}\|_2\\
\lesssim & \|\widehat{f}\|_2+\|\mathcal{D}_{\xi}^{\theta}(|\xi|^{\alpha}\widetilde{\psi})\psi\|_{2} \|\widehat{f}\|_{L^{\infty}}\\
\lesssim & \|\langle x \rangle f\|_{2}.
\end{aligned}    
\end{equation*}
Similarly, since $\psi=1$ in a neighborhood of the origin, we get from \eqref{negsteinderiv2}
\begin{equation*}
\begin{aligned}
\|\mathcal{D}_{\xi}^{\theta}\big(\frac{\widetilde{\psi}}{1+|\xi|^{\alpha}}\big)(1+(1-\alpha)|\xi|^\alpha)\psi \hat{f}\|_2 \lesssim & \|\widehat{f}\|_2+\|\mathcal{D}_{\xi}^{\theta}(|\xi|^{\alpha}\widetilde{\psi})\psi\|_{2} \|\widehat{f}\|_{L^{\infty}}\\
\lesssim & \|\langle x \rangle f\|_{2}.
\end{aligned}    
\end{equation*}
Using definition \eqref{d1b}, it is not hard to see
\begin{equation*}
\begin{aligned}
 \|\mathcal{D}_{\xi}^{\theta}\big(1+(1-\alpha)|\xi|^\alpha\psi \big)\hat{f}\|_2 \lesssim & \|\widehat{f}\|_2+\|\mathcal{D}_{\xi}^{\theta}(|\xi|^{\alpha}\psi)\widehat{f}\|_2\\
\lesssim & \|\langle x \rangle f\|_{2}. 
\end{aligned}    
\end{equation*}
We use property \eqref{Linftybound} to find

\begin{equation}\label{e1.23.1}
\begin{aligned}
\|\mathcal{D}_{\xi}^{\theta} \Big(\frac{1+(1-\alpha)|\xi|^\alpha}{(1+|\xi|^\alpha)^2}&(1-\psi)\Big)\widehat{f}\|_2 \\
\lesssim &\Big(\|\frac{1+(1-\alpha)|\xi|^\alpha}{(1+|\xi|^\alpha)^2}(1-\psi)\|_{L^{\infty}}\Big)\\
&+\|\partial_x\big(\frac{1+(1-\alpha)|\xi|^\alpha}{(1+|\xi|^\alpha)^2}(1-\psi)\big)\|_{L^{\infty}}\Big) \|\widehat{f}\|_2\\
\lesssim & \|f\|_{2}.
\end{aligned}    
\end{equation}
Collecting the previous estimates, we arrive at \eqref{conclestim1}.
\end{proof}

Using that $\wnt \leq \langle x\rangle^{\theta}$,  and that $0<\theta<\alpha-\frac{1}{2}<\alpha+\frac{1}{2}$, we can apply Claim \ref{ClaimA1} to get
\begin{equation*}\label{e1.36}
\begin{split}
|\ma_1|&\lesssim \|\wnt A_1\,(xu)\|_2\|x^2\wnt\,u\|_{2}\\
&\lesssim \|\langle x\rangle^{\theta} A_1\,(xu)\|_2\|x^2\wnt\,u\|_{2}\\
&\lesssim \|\langle x\rangle^{2} u\|_2\|x^2\wnt\,u\|_{2}.
\end{split}
\end{equation*}
Now, we take care of $\ma_2$. Cauchy-Schwarz inequality and the fact that for $0<\theta<\frac{1}{2}$ we know that $\wnt$ is a Muckenhoupt weight with $p=2$, so that $\mathcal H$ is a bounded operator in $L^2( (\wnt)^2\,  dx)$  with a constant independent of $N$ (see, Proposition 1 in \cite{FoPo}), it follows
\begin{equation*}\label{e1.40}
\begin{split}
|\ma_2|&\lesssim \|\wnt\hil D^{\alpha-1}A_3 u\|_2\|x^2\wnt\,u\|_{2}\\
    &\lesssim \|\wnt D^{\alpha-1}A_3 u\|_2\|x^2\wnt u\|_2\\
    &\lesssim (\|D^{\alpha-1}A_3 u\|_2+\||x|^\theta D^{\alpha-1}A_3\,u\|_2)\|x^2\wnt u\|_2\\
    &\lesssim (\ma_{2,1}+\ma_{2,2})\|x^2\wnt u\|_2.
\end{split}
\end{equation*}
Using that $A_3$ defines a bounded operator in $L^2$, and Sobolev embedding, we get
\begin{equation*}
 \begin{aligned}
\ma_{2,1}\lesssim & \||\xi|^{\alpha-1}\chi_{\{|\xi|\leq 1\}}\widehat{u}\|_2+\||\xi|^{\alpha-1}\chi_{\{|\xi|> 1\}}\widehat{u}\|_2\\
\lesssim & \||\xi|^{\alpha-1}\chi_{\{|\xi|\leq 1\}}\|_2\|\widehat{u}\|_{L^{\infty}}+\|\widehat{u}\|_2\\
\lesssim & \|\langle x \rangle u\|_2,
 \end{aligned}   
\end{equation*}
where we remark that $|\xi|^{\alpha-1}\chi_{\{|\xi|\leq 1\}}\in L^2(\mathbb{R})$ provided that $\alpha>\frac{1}{2}$.  In order to bound the term $\ma_{2,2},$ we proceed via Stein's derivative to find
\begin{equation*}\label{e1.41}
\begin{split}
   \ma_{2,2} &\lesssim \|\md^\theta (\frac{|\xi|^{\alpha-1}}{(1+|\xi|^\alpha)^3}\psi\hat{u})\|_2+\|\md^\theta (\frac{|\xi|^{\alpha-1}}{(1+|\xi|^\alpha)^3}(1-\psi)\hat{u})\|_2\\
   &\lesssim \ma_{2,2,1}+\ma_{2,2,2}.
   \end{split}
\end{equation*}
Let $\widetilde{\psi}$ be a smooth compactly supported function such that $\widetilde{\psi}\psi=\psi$. We apply Proposition \ref{negativeleibnizprop} with $g=|\xi|^{\alpha-1}\psi \widehat{u}$ to get
\begin{equation*}
\begin{aligned}
 \ma_{2,2,1}\lesssim \||\xi|^{\alpha-1}\psi \widehat{u}\|_{2}+\||\xi|^{\alpha-1}\psi \widehat{u}\mathcal{D}^{\theta}(|\xi|^{\alpha} \widetilde{\psi})\|_2+\|\mathcal{D}^{\theta}(|\xi|^{\alpha-1}\psi \widehat{u})\|_2  
\end{aligned}    
\end{equation*}
Let us estimate the above terms. The first one on the right-hand side of the previous inequality follows from the estimate for $\mathcal{A}_{2,1}$ above. Now, Proposition \ref{prop3}, the fact that $\alpha>\frac{1}{2}$ and that $2\alpha-\frac{1}{2}-\theta>0$ establish
\begin{equation*}
\begin{aligned}
\||\xi|^{\alpha-1}\psi \widehat{u}\mathcal{D}^{\theta}(|\xi|^{\alpha} \widetilde{\psi})\|_2 \lesssim & \||\xi|^{2\alpha-1-\theta}\psi \widehat{u}\|_2+\||\xi|^{\alpha-1}\psi \widehat{u}\|_2\\
\lesssim & \big(\||\xi|^{2\alpha-1-\theta}\psi \|_2+\||\xi|^{\alpha-1}\psi \|_2\big)\|\widehat{u}\|_{L^{\infty}}\\
\lesssim & \|\langle x\rangle u\|_2.
\end{aligned}    
\end{equation*}
By \eqref{pointwise2},  Proposition \ref{prop4}, we obtain that
   \begin{equation*}
\begin{split}
   \|\mathcal{D}^{\theta}(|\xi|^{\alpha-1}\psi \widehat{u})\|_2	 &\lesssim \|\md^\theta (|\xi|^{\alpha-1}\psi) \hat{u}\|_2+\||\xi|^{\alpha-1}\psi \md^\theta  \hat{u}\|_2\\
	&\lesssim  \|\md^\theta (|\xi|^{\alpha-1}\psi) \chi_{\{|\xi|\leq1\}}\hat{u}\|_2+\|\md^\theta (|\xi|^{\alpha-1}\psi)\chi_{\{|\xi|\geq1\}}\hat{u}\|_2\\
	&+\||\xi|^{\alpha-1}\psi\|_2\|\md^\theta \hat{u}\|_{L^\infty}\\
	&\lesssim \|\frac{\chi_{\{|\xi|\leq1\}}}{|\xi|^{1-\alpha+\theta}}\hat{u}\|_2 + \|\frac{\chi_{\{|\xi|\geq1\}}}{|\xi|^{1-\alpha+\theta}}\hat{u}\|_2+\|\md^\theta \hat{u}\|_{L^\infty}\\
	&\lesssim \|\frac{\chi_{\{|\xi|\leq1\}}}{|\xi|^{1-\alpha+\theta}}\|_2\|\hat{u}\|_{L^\infty}+\|\frac{\chi_{\{|\xi|\geq1\}}}{|\xi|^{1-\alpha+\theta}}\|_{L^\infty}\|\hat{u}\|_2+\|\langle x\rangle^2 u\|_2\\
	&\lesssim \|u\|_{s,2}+\|\langle x\rangle^2 u\|_2,
\end{split}
\end{equation*}
where we have used \eqref{Linftybound} and Sobolev embedding to control $\|\md^\theta \hat{u}\|_{L^\infty}$. We now proceed to estimate $\ma_{2,2,2}$. By using \eqref{pointwise2}, and \eqref{Linftybound}, we get
  \begin{equation*}
\begin{split}
   	\ma_{2,2,2}	 &\lesssim \|\md^\theta (\frac{|\xi|^{\alpha-1}(1-\psi)}{(1+|\xi|^\alpha)^3}) \hat{u}\|_2+\|\frac{|\xi|^{\alpha-1}(1-\psi)}{(1+|\xi|^\alpha)^3}\md^\theta  \hat{u}\|_2\\
   	&\lesssim  \|\md^\theta (\frac{|\xi|^{\alpha-1}(1-\psi)}{(1+|\xi|^\alpha)^3})\|_{L^{\infty}}\|\hat{u}\|_2+\||\xi|^{\alpha-1}(1-\psi)\|_{L^\infty}\|\md^\theta \hat{u}\|_2\\
	&\lesssim \Big(\|\frac{|\xi|^{\alpha-1}(1-\psi)}{(1+|\xi|^\alpha)^3}\|_{L^{\infty}}+\|\partial_x (\frac{|\xi|^{\alpha-1}(1-\psi)}{(1+|\xi|^\alpha)^3})\|_{L^{\infty}}\Big)\|\hat{u}\|_2+\|\md^\theta \hat{u}\|_2\\
	&\lesssim \|u\|_{s,2}+\|\langle x\rangle^2 u\|_2.
\end{split}
\end{equation*}
This completes the estimate of $\mathcal{A}_{2,2}$ and in turn the study of $\mathcal{A}_2$. In order to  handle the integral term $\ma_3$, we use Cauchy-Schwarz inequality and the continuity of the Hilbert transform with respect to the weight $\wnt$ when $0<\theta<1/2$ to find
\begin{equation*}
\begin{split}
|\ma_3|&\lesssim \|\wnt\hil D^{2\alpha-1}A_3 u\|_2\|x^2\wnt\,u\|_{2}\\
    &\lesssim \|\wnt D^{2\alpha-1}A_3 u\|_2\|x^2\wnt u\|_2\\
    &\lesssim (\|D^{2\alpha-1}A_3 u\|_2+\||x|^\theta D^{2\alpha-1}A_3\,u\|_2)\|x^2\wnt u\|_2\\
    &\lesssim (\ma_{3,1}+\ma_{3,2})\|x^2\wnt u\|_2.
\end{split}
\end{equation*}

It is apparent that $\ma_{3,1}\lesssim \|u\|_2 $ for $\alpha>1/2.$ On the other hand for $\ma_{3,2}$ we have
\begin{equation*}
\begin{split}
   \ma_{3,2} &\lesssim \|\md^\theta (\frac{|\xi|^{2\alpha-1}}{(1+|\xi|^\alpha)^3}\psi\hat{u})\|_2+\|\md^\theta (\frac{|\xi|^{2\alpha-1}}{(1+|\xi|^\alpha)^3}(1-\psi)\hat{u})\|_2\\
   &\lesssim \ma_{3,2,1}+\ma_{3,2,2}.
   \end{split}
\end{equation*}
The estimate for $\ma_{3,2,2}$ is similar to that of $\ma_{2,2,2}$, thus similar arguments show
  \begin{equation*}
\begin{split}
   	\ma_{3,2,2}	
	&\lesssim \|u\|_{s,2}+\|\langle x\rangle^2 u\|_2.
\end{split}
\end{equation*}
The estimate for $\ma_{3,2,1}$ resembles that of $\ma_{2,2,1}$. For completeness, let us show such an estimate.  We first apply Proposition \ref{negativeleibnizprop} to deduce
\begin{equation*}
\begin{aligned}
 \ma_{3,2,1}\lesssim \||\xi|^{2\alpha-1}\psi \widehat{u}\|_{2}+\||\xi|^{2\alpha-1}\psi \widehat{u}\mathcal{D}^{\theta}(|\xi|^{\alpha} \widetilde{\psi})\|_2+\|\mathcal{D}^{\theta}(|\xi|^{2\alpha-1}\psi \widehat{u})\|_2,  
\end{aligned}    
\end{equation*}
where $\widetilde{\psi}$ is a smooth function compactly supported such that $\widetilde{\psi}\psi=\psi$. Since $2\alpha-\frac{1}{2}>0$, the estimate for the first factor on the right-hand side of the expression above is similar to that of $\mathcal{A}_{2,1}$ above. Next, Proposition \ref{prop3} shows
\begin{equation*}
\begin{aligned}
\||\xi|^{2\alpha-1}\psi \widehat{u}\mathcal{D}^{\theta}(|\xi|^{\alpha} \widetilde{\psi})\|_2 \lesssim & \||\xi|^{3\alpha-1-\theta}\psi \widehat{u}\|_2+\||\xi|^{2\alpha-1}\psi \widehat{u}\|_2\\
\lesssim & \big(\||\xi|^{3\alpha-1-\theta}\psi \|_2+\||\xi|^{2\alpha-1}\psi \|_2\big)\|\widehat{u}\|_{L^{\infty}}\\
\lesssim & \|\langle x\rangle u\|_2,
\end{aligned}    
\end{equation*}
which holds provided that $3\alpha-\frac{1}{2}-\theta>0$, and $2\alpha-\frac{1}{2}>0$. Since $2\alpha-1>0$, in this case, we apply Proposition \ref{prop3} instead of Proposition \ref{prop4} to get
\begin{equation*}\label{e1.43}
\begin{split}
   	\|\mathcal{D}^{\theta}(|\xi|^{2\alpha-1}\psi \widehat{u})\|_2 
	\lesssim &  \|\md^\theta (|\xi|^{2\alpha-1}\psi) \chi_{\{|\xi|\leq1\}}\hat{u}\|_2+\|\md^\theta (|\xi|^{2\alpha-1}\psi)\chi_{\{|\xi|\geq1\}}\hat{u}\|_2\\
	&+\||\xi|^{2\alpha-1}\psi\|_{L^\infty}\|\md^\theta \hat{u}\|_{2}\\
	\lesssim & \|(|\xi|^{2\alpha-1-\theta}+1)\chi_{\{|\xi|\leq1\}}\hat{u}\|_2 + \|\frac{\chi_{\{|\xi|\geq1\}}}{|\xi|^{1/2+\theta}}\hat{u}\|_2+\|\md^\theta \hat{u}\|_{2}\\
	\lesssim & \|(|\xi|^{2\alpha-1-\theta}+1)\chi_{\{|\xi|\leq1\}}\|_{L^\infty}\|\hat{u}\|_{2}+\|\frac{\chi_{\{|\xi|\geq1\}}}{|\xi|^{1/2+\theta}}\|_{2}\|\hat{u}\|_{L^\infty}\\
 &+\|\langle x\rangle^{\theta} u\|_2\\
	\lesssim & \|u\|_{s,2}+\|\langle x\rangle^2 u\|_2.
\end{split}
\end{equation*}
This finishes the estimate of $\mathcal{A}_3$. In the case of the  $\ma_4$ term, Cauchy-Schwarz inequality and Lemma \ref{commA} imply that
 \begin{equation*}\label{e1.45}
\begin{split}
|\ma_4|&\lesssim \|[\wnt;A](x^2u)\|_2\|x^2\wnt\,u\|_{2}\\
	&\lesssim \|x^2u\|_2\|x^2\wnt\,u\|_{2}\\
	&\lesssim \|\langle x\rangle^2u\|_2\|x^2\wnt\,u\|_{2}.\\
\end{split}
\end{equation*}
Since $\mathcal{B}$ involves the same operators as those studied in $\mathcal{A}$, by changing $u$ by $u^2$ in the argument of the operator $A$ in our previous analysis, we get 
\begin{equation*}
\begin{aligned}
 \mathcal{B}\lesssim & \big(\|\langle x\rangle^2 u^2\|_2+\|u^2\|_{s,2}\big)\|x^2\wnt\,u\|_2+\|J^{1-\alpha}(\langle x\rangle^2\wnt\,u^2)\|_2\|x^2\wnt\,u\|_2\\
 \lesssim & \big(\|u\|_{s,2}\|\langle x\rangle^2 u\|_2+\|u\|_{s,2}^2\big)\|x^2\wnt\,u\|_2\\
 &+\|u\|_{s,2}\big(\|\langle x\rangle^2 u\|_2+\|x^2\wnt\,u\|_2\big)\|x^2\wnt\,u\|_2,
\end{aligned}    
\end{equation*}
where we remark that substituting $u$ by $u^2$, $\mathcal{A}_5$ is transferred into $\mathcal{B}_5=\int A(\wnt x^2u^2)\,\wnt x^2u\,dx$, which this time it is not necessarily equal to zero but its estimate follows from the same arguments in the study of $\mathcal{B}_{2,2}$ in the proof of Proposition \ref{propTH1}, whose validity is given by interpolation \eqref{complex2}.

Hence, it follows from these previous computations that the energy estimate for these weights can be stated as   
\begin{equation*}\label{e1.42}
\begin{aligned}
\frac12 \frac{d}{dt}\int (x^2\wnt\,u)^2\,dx\lesssim & 
\Big(\|\langle x\rangle^2 u\|_2+\|u\|_{s,2}+\|u\|_{s,2}\|\langle x\rangle^2 u\|_2\\
&+\|u\|_{s,2}^2+\|u\|_{s,2}\|x^2\wnt\,u\|_{2}\Big)\|x^2\wnt\,u\|_{2}.    
\end{aligned}
\end{equation*}
From which by applying Gronwall's inequality and taking $N\to \infty$, the desired result follows.
\end{proof}


\section{proof of Theorem \ref{Main3}}\label{SectUnique}

We begin with some estimates for the linear group $\{e^{tA}\}$ in weighted spaces. We remark that the following result is a consequence of the proof of Theorem \ref{Main2} above.

\begin{proposition}\label{Propwightedineq}
Let $\alpha \in (0,1)$ and $r\in (0,\frac{3}{2}+\alpha)$. If $\varphi \in L^2(\mathbb{R})\cap L^2(|x|^{2r}\, dx)=Z_{0,r}$, then for each $t\in \mathbb{R}$
\begin{equation}\label{groupineq}
 \begin{aligned}
 \|\langle x\rangle^{r} e^{t A}\varphi\|_2\lesssim  \langle t \rangle^{\lceil r \rceil}\big\|\langle x\rangle^r \varphi\|_2,  
 \end{aligned}   
\end{equation}
where $\lceil r \rceil$ stands for the ceiling function at $r$, which maps $r$ to the least integer greater than or equal to $r$.
    
\end{proposition}

\begin{proof}
We may assume that $\varphi$ is sufficiently regular to perform all the arguments shown below. The general case follows from approximation to our estimates. We write $r=j+\theta$, where $j\geq 0$ is an  integer and  $\theta\in [0,1]$. Since $u(t)=e^{tA}\varphi$ is a solution of the equation $\partial_t u-Au=0$, we multiply such equation by $(x^{j}\wnt)^2 u$ to get
 \begin{equation}\label{diffeq2}
 \frac{1}{2}\frac{d}{dt}\int (x^j\wnt u)^2\, dx-\int A u (x^{j}\wnt)^2 u=0.    
 \end{equation}   
When $j=0$, $\theta\in (0,1]$, the estimate for the factor $\mathcal{A}$ in Proposition \ref{propTH1} and \eqref{diffeq2} show
\begin{equation*}
  \frac{d}{dt}\int (\wnt u)^2\, dx\lesssim \|u\|_2\|\wnt u\|_2\\
  \lesssim \|\varphi\|_2\|\wnt u\|_2,
\end{equation*}
where we have used that $\{e^{tA}\}$ is a group of isometries in $L^2$.  Hence, Gronwall's inequality and taking $N\to \infty$ establish
\begin{equation*}
    \|\langle x\rangle^{\theta} u(t)\|_2\lesssim \|\langle x\rangle^{\theta} \varphi\|_2+|t|\|\varphi\|_2, 
\end{equation*}
for all $\theta\in (0,1]$. This yields \eqref{groupineq} for the case $r=\theta\in(0,1]$. By increasing $j\in \mathbb{Z}$, $j\geq 0$, and keeping the condition  $j+\theta<\frac{3}{2}+\alpha$, we can iterate the previous argument to get the desired result. More precisely, when $j=1$ and $\theta\in (0,1]$, plugging the estimate made for $\mathcal{A}$ in Proposition \ref{propTH2} into \eqref{diffeq2}, gives us a differential inequality involving $\|x\wnt u\|_2$ for which $N\to\infty$ yields \eqref{groupineq}. Finally, arguing as before, when $j=2$ (which only happens when $\alpha\in (\frac{1}{2},1)$), and $\theta\in (0,\alpha-\frac{1}{2})$, we use the estimate $\mathcal{A}$ in Proposition \ref{propTH3} to deduce \eqref{groupineq}.
\end{proof}

\begin{proof}[Proof of Theorem \ref{Main3}]
 We restrict our proof  of Theorem \ref{Main3} to  the lower dispersion case $\alpha\in (0,\frac{1}{2})$. When $\alpha\in [\frac{1}{2},1)$,  similar estimates in the former case and identity \eqref{der2} yield the desired conclusion. Moreover, without loss of generality, we will assume that $t_1=0$, i.e., we consider $u\in C([0,T];Z_{s,r})$ to be a solution of \eqref{fBBM} with initial condition $\varphi$, where $s>s_{\alpha}$ and $r=\big(\frac{3}{2}+\alpha\big)^{-}$ (our results below make clear the admissible values for $r$), such that 
\begin{equation*}
 u(t_2), \varphi \in L^2(|x|^{2\big(\frac{3}{2}+\alpha\big)}\, dx),  
\end{equation*}
for some $0<t_2\leq T$. Applying the Fourier transform, the previous statement is equivalent to have
\begin{equation*}
 \widehat{u}(\xi,t_2), \widehat{\varphi}(\xi) \in H^{\frac{3}{2}+\alpha}(\mathbb{R}),  
\end{equation*}
which is also equivalent to
\begin{equation}\label{assumpt1}
\partial_{\xi}\widehat{u}(\xi,t_2), \partial_{\xi}\widehat{\varphi}(\xi) \in H^{\frac{1}{2}+\alpha}(\mathbb{R}).  
\end{equation}
Now, the fact that $u\in C([0,T];Z_{s,r})$ with $r>1$ solves the integral equation \eqref{IE} allow us to compute $\partial_{\xi} \widehat{u}(t_2)$ to write
\begin{equation*}
\begin{aligned}
\partial_{\xi} \widehat{u}(t_2)=&-it_2\frac{1+(1-\alpha)|\xi|^{\alpha}}{(1+|\xi|^{\alpha})^2}F(t_2,\xi)\widehat{\varphi}(\xi)+ F(t_2,\xi)\partial_{\xi}\widehat{\varphi}(\xi) \\
&-i\int_0^{t_2} (t_2-\tau)\frac{1+(1-\alpha)|\xi|^{\alpha}}{(1+|\xi|^{\alpha})^2}F(t_2-\tau,\xi)\widehat{A u^2}(\xi,\tau)\, d\tau\\
&+\int_0^{t_2} F(t_2-\tau,\xi)\partial_{\xi}\widehat{A u^2}(\xi,\tau)\, d\tau \\
=:&\mathcal{U}_1+\mathcal{U}_2+\mathcal{U}_3+\mathcal{U}_4,
\end{aligned}    
\end{equation*}
where we also used identity \eqref{der1}. By Plancherel's identity and Proposition \ref{Propwightedineq}, we have
\begin{equation*}
\begin{aligned}
\|J^{\frac{1}{2}+\alpha}_{\xi}(\mathcal{U}_2)\|_2=\|\langle x \rangle^{\frac{1}{2}+\alpha}e^{t_2A}(x\varphi)\|_2\lesssim  \langle t_2 \rangle^{\lceil \frac{1}{2}+\alpha \rceil}\|\langle x\rangle^{\frac{3}{2}+\alpha}\varphi\|_2.  
\end{aligned}    
\end{equation*}
Next, using that the symbol of the operator $A$ is $-i\xi(1+|\xi|^{\alpha})^{-1}$, we write
\begin{equation*}
\begin{aligned}
\mathcal{U}_3=-\int_0^{t_2} (t_2-\tau)\frac{\xi+(1-\alpha)|\xi|^{\alpha}\xi}{(1+|\xi|^{\alpha})^3\langle \xi \rangle^{1-\alpha}}F(t_2-\tau,\xi)\widehat{J^{1-\alpha}(u^2)}(\xi,\tau)\, d\tau.   
\end{aligned}    
\end{equation*}
Then, applying Theorem \ref{theorem9b}, property \eqref{pointwise2}, and \eqref{Linftybound}, we get
\begin{equation}\label{eqU3}
\begin{aligned}
\|J^{\frac{1}{2}+\alpha}_{\xi} &(\mathcal{U}_3)\|_2\\
\lesssim &  \langle t_2 \rangle \Big(\|\frac{\xi+(1-\alpha)|\xi|^{\alpha}\xi}{(1+|\xi|^{\alpha})^3\langle \xi \rangle^{1-\alpha}}\|_{L^{\infty}}+\|\partial_{\xi}\Big(\frac{\xi+(1-\alpha)|\xi|^{\alpha}\xi}{(1+|\xi|^{\alpha})^3\langle \xi \rangle^{1-\alpha}}\Big)\|_{L^{\infty}}\Big)\\
&\qquad \times \int_0^{t_2} \|J_{\xi}^{\frac{1}{2}+\alpha}\big(F(t_2-\tau,\xi)\widehat{J^{1-\alpha}(u^2)}(\xi,\tau)\big)\|_2\, d\tau\\
\lesssim & \langle t_2 \rangle \int_0^{t_2} \|\langle x\rangle^{\frac{1}{2}+\alpha}e^{(t_2-\tau)A}J^{(1-\alpha)}(u^2)(\tau)\|_2\, d\tau,
\end{aligned}    
\end{equation}
where we also applied Plancherel's identity. Consequently, Proposition \ref{Propwightedineq} and interpolation (see Remark \ref{remarkinterp}) show
\begin{equation*}
\begin{aligned}
\|J^{\frac{1}{2}+\alpha}(\mathcal{U}_3)\|_2\lesssim & \langle t_2 \rangle^{\lceil \frac{1}{2}+\alpha \rceil+1}\int_0^{t_2}\|\langle x\rangle^{\frac{1}{2}+\alpha}J^{1-\alpha}(u^2)(\tau)\|_2\, d\tau\\
\lesssim & \langle t_2 \rangle^{\lceil \frac{1}{2}+\alpha \rceil+1}\int_0^{t_2}\|\langle x\rangle^{1+2\alpha}u^2(\tau)\|_2^{\frac{1}{2}}\|J^{2(1-\alpha)}(u^2)(\tau)\|_2^{\frac{1}{2}}\, d\tau. 
\end{aligned}    
\end{equation*}
Our local theory in $H^{s}(\mathbb{R})$, $s>s_{\alpha}=2(1-\alpha)$ for the IVP \eqref{fBBM}, and Lemma \ref{leibnizhomog} imply $\sup_{t\in [0,T]}\|J^{2(1-\alpha)}(u^2)(\tau)\|_2 <\infty$.  Moreover, since $1+2\alpha\leq \frac{3}{2}+\alpha$, we use Sobolev embedding, interpolation, and the fact that $u\in C([0,T];H^s(\mathbb{R}))$ with $s>s_{\alpha}$ to get
\begin{equation}\label{eqinterpnon1}
\begin{aligned}
 \|\langle x\rangle^{1+2\alpha}u^2\|_2\leq \|\langle x\rangle^{\frac{3}{2}+\alpha}u^2\|_2=&\|\langle x\rangle^{\frac{3+2\alpha}{4}}u\|_{L^4}^2\\
 \lesssim & \|J^{\frac{1}{4}}(\langle x\rangle^{\frac{3+2\alpha}{4}}u)\|_{2}^2 \\
 \lesssim & \|\langle x \rangle^{\frac{(3+2\alpha)s}{4s-1}} u\|_2^{\frac{4s-1}{2s}}\|u\|_{s,2}^{\frac{1}{2s}},
\end{aligned}    
\end{equation}
which is controlled provided that $u\in C([0,T];Z_{s,r})$ with $\max\{1,\frac{(3+2\alpha)s}{4s-1}\}< r<\frac{3}{2}+\alpha$. Hence, we have proved that  $\mathcal{U}_3\in H^{\frac{1}{2}+\alpha}(\mathbb{R})$. Summarizing, 
\begin{equation*}
\partial_{\xi}\widehat{u}(t_2)\in H^{\frac{1}{2}+\alpha}(\mathbb{R}) \qquad \text{ if and only if } \qquad \mathcal{U}_1+\mathcal{U}_4\in  H^{\frac{1}{2}+\alpha}(\mathbb{R}).   
\end{equation*}
Next, we further decompose the terms $\mathcal{U}_1$ and $\mathcal{U}_4$. To do so, let $\psi$ be a smooth function compactly supported such that $\psi(\xi)=1$, whenever $|\xi|\leq 1$. We write
\begin{equation*}
\begin{aligned}
\mathcal{U}_1=& -it_2\frac{1+(1-\alpha)|\xi|^{\alpha}}{(1+|\xi|^{\alpha})^2}(F(t_2,\xi)-1)\widehat{\varphi}(\xi)\\
&-it_2\frac{1+(1-\alpha)|\xi|^{\alpha}}{(1+|\xi|^{\alpha})^2}(1-\psi)\widehat{\varphi}(\xi)\\
&-it_2\frac{1+(1-\alpha)|\xi|^{\alpha}}{(1+|\xi|^{\alpha})^2}(\widehat{\varphi}(\xi)-\widehat{\varphi}(0))\psi\\
&-it_2\frac{1+(1-\alpha)|\xi|^{\alpha}}{(1+|\xi|^{\alpha})^2}\widehat{\varphi}(0)\psi\\
=:&\mathcal{U}_{1,1}+\mathcal{U}_{1,2}+\mathcal{U}_{1,3}+\mathcal{U}_{1,4}.
\end{aligned}    
\end{equation*}
We divide the estimate for $\mathcal{U}_4$ as follows
\begin{equation*}
\begin{aligned}
  \mathcal{U}_4=&-i \int_0^{t_2} F(t_2-\tau,\xi)\frac{\xi}{1+|\xi|^{\alpha}}\partial_{\xi}\widehat{u^2}(\xi,\tau)\, d\tau\\
 &-i\int_0^{t_2} \frac{1+(1-\alpha)|\xi|^{\alpha}}{(1+|\xi|^{\alpha})^2}F(t_2-\tau,\xi)\widehat{u^2}(\xi,\tau)\, d\tau\\
 =&\mathcal{U}_{4,1}+\mathcal{U}_{4,2}.
\end{aligned}
\end{equation*}
and
\begin{equation*}
\begin{aligned}
\mathcal{U}_{4,2}=& -i\frac{1+(1-\alpha)|\xi|^{\alpha}}{(1+|\xi|^{\alpha})^2}\int_0^{t_2}(F(t_2-\tau,\xi)-1)\widehat{u^2}(\xi,\tau)\, d\tau\\
&-i\frac{1+(1-\alpha)|\xi|^{\alpha}}{(1+|\xi|^{\alpha})^2}\int_0^{t_2}(1-\psi)\widehat{u^2}(\xi,\tau )\, d\tau\\
&-i\frac{1+(1-\alpha)|\xi|^{\alpha}}{(1+|\xi|^{\alpha})^2}\int_0^{t_2}(\widehat{u^2}(\xi,\tau)-\widehat{u^2}(0,\tau))\psi\, d\tau\\
&-i\frac{1+(1-\alpha)|\xi|^{\alpha}}{(1+|\xi|^{\alpha})^2}\int_0^{t_2}\widehat{u^2}(0,\tau)\psi\, d\tau\\
=:&\mathcal{U}_{4,2,1}+\mathcal{U}_{4,2,2}+\mathcal{U}_{4,2,3}+\mathcal{U}_{4,2,4}.
\end{aligned}    
\end{equation*}
We have the following result
\begin{proposition}\label{propuniquecon} Let $\alpha\in (0,\frac{1}{2})$, $s>s_{\alpha}$, $\max\{1,\frac{(3+2\alpha)s}{4s-1},\frac{s(3+2\alpha)}{4(s+\alpha-1)}\}<r<\frac{3}{2}+\alpha$,  $\varphi\in Z_{s,\frac{3}{2}+\alpha}$, and $u\in C([0,T];Z_{s,r})$. Then
\begin{equation*}
  \mathcal{U}_{4,1}, \mathcal{U}_{1,m}, \mathcal{U}_{4,2,m}\in H^{\frac{1}{2}+\alpha}(\mathbb{R}),  
\end{equation*}
for each $m=1,2,3$.  
\end{proposition}

We deduce the previous result later. For now, assuming the validity of Proposition \ref{propuniquecon}, we have that $\partial_{\xi}\widehat{u}(t_2)\in H^{\frac{1}{2}+\alpha}(\mathbb{R})$ if and only if $\mathcal{U}_{1,4}+\mathcal{U}_{4,2,4}\in  H^{\frac{1}{2}+\alpha}(\mathbb{R})$, which is equivalent to
\begin{equation}\label{conclueq}
\begin{aligned}
-i\frac{1+(1-\alpha)|\xi|^{\alpha}}{(1+|\xi|^{\alpha})^2}\psi\Big(t_2\widehat{\varphi}(0)+\int_0^{t_2}\widehat{u^2}(0,\tau)\, d\tau \Big)\in H^{\alpha+\frac{1}{2}}(\mathbb{R}).   
\end{aligned}    
\end{equation}
We need the following result.
\begin{proposition}\label{propuniquecon2} Let $\alpha\in (0,\frac{1}{2})$, then 
\begin{equation*}
\frac{1+(1-\alpha)|\xi|^{\alpha}}{(1+|\xi|^{\alpha})^2}\psi\notin H^{\frac{1}{2}+\alpha}(\mathbb{R}).
\end{equation*}  
\end{proposition}
Once again, to give continuity to our argument, we will show the previous result at the end of this section. Since \eqref{conclueq} holds, Proposition \ref{propuniquecon2} forces us to have
\begin{equation*}
t_2\widehat{\varphi}(0)+\int_0^{t_2}\widehat{u^2}(0,\tau)\, d\tau=0,
\end{equation*}
which because $\tau \mapsto \|u(\tau)\|_2^2$ is a continuous function, allows us to complete the proof of Theorem \ref{Main3}.
\end{proof}

\begin{proof}[Proof of Proposition \ref{propuniquecon}]
Let us show first that $\mathcal{U}_{1,m}\in H^{\frac{1}{2}+\alpha}(\mathbb{R})$ for each $m=1,2,3$.  We write $\mathcal{U}_{1,1}=\mathcal{U}_{1,1}\psi+\mathcal{U}_{1,1}(1-\psi)$. Since $\frac{1+(1-\alpha)|\xi|^{\alpha}}{(1+|\xi|^{\alpha})^2}(F(t_2,\xi)-1)(1-\psi)$ defines a smooth bounded function with all its derivatives bounded, we can use property \eqref{Linftybound} to argue as in \eqref{e1.23.1} to control $\mathcal{U}_{1,1}(1-\psi)$. The same idea applies to estimate $\mathcal{U}_{1,2}$. Summarizing, we get
\begin{equation*}
 \|\mathcal{D}^{\frac{1}{2}+\alpha}\big(\mathcal{U}_{1,1}(1-\psi)\big)\|_2 + \|\mathcal{D}^{\frac{1}{2}+\alpha}( \mathcal{U}_{1,2})\|_2 \lesssim \|\langle x \rangle^{\frac{1}{2}+\alpha}\varphi\|_2,   
\end{equation*}
where the implicit constant depends on $t_2$. Next, we apply Proposition \ref{negativeleibnizprop} to deduce
\begin{equation*}
\begin{aligned}
\|\mathcal{D}^{\frac{1}{2}+\alpha}\big(\mathcal{U}_{1,1}\psi\big)\|_2 \lesssim & \|(1+(1-\alpha)|\xi|^{\alpha}) (F(t_2,\xi)-1)\psi \widehat{\varphi}\|_2\\
&+\|(1+(1-\alpha)|\xi|^{\alpha}) (F(t_2,\xi)-1)\psi \widehat{\varphi}\mathcal{D}^{\frac{1}{2}+\alpha}(|\xi|^{\alpha}\widetilde{\psi})\|_2 \\
&+\|\mathcal{D}^{\frac{1}{2}+\alpha}\big((1+(1-\alpha)|\xi|^{\alpha}) (F(t_2,\xi)-1)\psi \widehat{\varphi}\big)\|_2. 
\end{aligned}    
\end{equation*}
Let us estimate each of the factors on the right-hand side of the inequality above. We first observe
\begin{equation*}
\begin{aligned}
\|(1+(1-\alpha)|\xi|^{\alpha}) (F(t_2,\xi)-1)\psi \widehat{\varphi}\|_2 \lesssim & \|\langle \xi\rangle^{\alpha}\psi\|_{L^{\infty}}\|\widehat{\varphi}\|_2\\
\lesssim & \|\varphi\|_2.
\end{aligned}    
\end{equation*}
Using the mean value inequality and then Sobolev embedding
\begin{equation*}
\begin{aligned}
\|(1+(1-\alpha)|\xi|^{\alpha}) &(F(t_2,\xi)-1)\psi \widehat{\varphi}\mathcal{D}^{\frac{1}{2}+\alpha}(|\xi|^{\alpha}\widetilde{\psi})\|_2\\
\lesssim & \|\langle \xi \rangle^{\alpha} \frac{\xi}{1+|\xi|^{\alpha}}\psi \widehat{\varphi}\mathcal{D}^{\frac{1}{2}+\alpha}(|\xi|^{\alpha}\widetilde{\psi})\|_2\\
\lesssim & \|\langle \xi \rangle^{\alpha} \xi \psi \mathcal{D}^{\frac{1}{2}+\alpha}(|\xi|^{\alpha}\widetilde{\psi})\|_2\|\widehat{\varphi}\|_{L^{\infty}}\\
\lesssim & \|\langle x \rangle^{\frac{1}{2}+\alpha}\varphi\|_2,
\end{aligned}    
\end{equation*}
where we have used Proposition \ref{prop3} to conclude that $\langle \xi \rangle^{\alpha} \xi \psi \mathcal{D}^{\frac{1}{2}+\alpha}(|\xi|^{\alpha}\widetilde{\psi})\in L^2(\mathbb{R})$. In this part, the extra weight $\xi$ is useful to allow such integration.

Next, we write
\begin{equation*}
F(t_2,\xi)-1=-i\frac{\xi}{1+|\xi|^{\alpha}} \int_0^{t_2} F(\sigma,\xi)\, d\sigma.  
\end{equation*}
Then, using \eqref{pointwise2a}, and the linearity of the operator $D^{\frac{1}{2}+\alpha}$, we get
\begin{equation*}
\begin{aligned}
\|\mathcal{D}^{\frac{1}{2}+\alpha}&\big((1+(1-\alpha)|\xi|^{\alpha}) (F(t_2,\xi)-1)\psi \widehat{\varphi}\big)\|_2 \\
\lesssim & \int_0^{t_2}\|D^{\frac{1}{2}+\alpha}\big(\frac{(1+(1-\alpha)|\xi|^{\alpha})}{1+|\xi|^{\alpha}}\xi F(\sigma,\xi)\psi \widehat{\varphi}\big)\|_2\, d\sigma \\
\lesssim & \int_0^{t_2}\|\mathcal{D}^{\frac{1}{2}+\alpha}\big(\frac{(1+(1-\alpha)|\xi|^{\alpha})}{1+|\xi|^{\alpha}}\xi F(\sigma,\xi)\psi \widehat{\varphi}\big)\|_2\, d\sigma\\
=:& \int_0^{t_2} \mathcal{I}_1\, d\sigma.
\end{aligned}    
\end{equation*}
To complete the preceding estimate, we apply Proposition \ref{negativeleibnizprop} to infer
\begin{equation*}
\begin{aligned}
\mathcal{I}_1\lesssim &\|(1+(1-\alpha)|\xi|^{\alpha})\xi F(\sigma,\xi)\psi \widehat{\varphi}\|_2\\
&+\|(1+(1-\alpha)|\xi|^{\alpha})\xi F(\sigma,\xi)\psi \widehat{\varphi}\mathcal{D}^{\frac{1}{2}+\alpha}(|\xi|^{\alpha}\widetilde{\psi})\|_2\\
&+\|\mathcal{D}^{\frac{1}{2}+\alpha}\big((1+(1-\alpha)|\xi|^{\alpha})\xi F(\sigma,\xi)\psi \widehat{\varphi}\big)\|_2\\
=:&\mathcal{I}_{1,1}+\mathcal{I}_{1,2}+\mathcal{I}_{1,3}.
\end{aligned}    
\end{equation*}
Similar arguments to the ones used above give us
\begin{equation*}
 \mathcal{I}_{1,1}+\mathcal{I}_{1,1}\lesssim \|\langle x\rangle^{\frac{1}{2}+\alpha}\varphi\|_2,   
\end{equation*}
where the implicit constant is independent of $\sigma$. We proceed with $\mathcal{I}_{1,3}$. Now, using \eqref{pointwise2}, one gets
\begin{equation*}
\begin{aligned}
\mathcal{I}_{1,3}\lesssim &  \|\mathcal{D}^{\frac{1}{2}+\alpha}\big(\xi F(\sigma,\xi)\psi \widehat{\varphi}\big)\|_2+\|\mathcal{D}^{\frac{1}{2}+\alpha}\big(|\xi|^{\alpha}\xi F(\sigma,\xi)\psi \widehat{\varphi}\big)\|_2 \\
\lesssim & \|\mathcal{D}^{\frac{1}{2}+\alpha}\big(\xi \psi \big)F(\sigma,\xi)\widehat{\varphi}\|_2+\|\xi \psi\mathcal{D}^{\frac{1}{2}+\alpha}(F(\sigma,\xi)\widehat{\varphi})\|_2\\
&+\|\mathcal{D}^{\frac{1}{2}+\alpha}\big(|\xi|^{\alpha} \xi \psi\big)F(\sigma,\xi)\widehat{\varphi}\|_2+\||\xi|^{\alpha}\xi \psi\mathcal{D}^{\frac{1}{2}+\alpha}(F(\sigma,\xi)\widehat{\varphi})\|_2.
\end{aligned}    
\end{equation*}
Now, Proposition \ref{prop3} shows
\begin{equation*}
\begin{aligned}
\|\mathcal{D}^{\frac{1}{2}+\alpha}\big(\xi \psi \big)F(\sigma,\xi)\widehat{\varphi}\|_2+\|\mathcal{D}^{\frac{1}{2}+\alpha}\big(|\xi|^{\alpha} \xi \psi\big)F(\sigma,\xi)\widehat{\varphi}\|_2 \lesssim \|\langle x\rangle^{\frac{1}{2}+\alpha}\varphi\|_2,
\end{aligned}    
\end{equation*}
and since $\xi\psi, \xi|\xi|^{\alpha}\psi \in L^{\infty}(\mathbb{R})$, we apply Plancherel's identity and Proposition \ref{Propwightedineq} to deduce
\begin{equation*}
\begin{aligned}
\|\xi \psi\mathcal{D}^{\frac{1}{2}+\alpha}(F(\sigma,\xi)\widehat{\varphi})\|_2&+\||\xi|^{\alpha}\xi \psi\mathcal{D}^{\frac{1}{2}+\alpha}(F(\sigma,\xi)\widehat{\varphi})\|_2\\ 
\lesssim & \|\langle x\rangle^{\frac{1}{2}+\alpha}e^{\sigma A}\varphi\|_2\\
\lesssim & \langle \sigma \rangle^{\lceil \frac{1}{2}+\alpha \rceil}\|\langle x\rangle^{\frac{1}{2}+\alpha} \varphi\|_2. 
\end{aligned}    
\end{equation*}
This completes the estimate for $\mathcal{J}_{1,3}$, and it follows that
\begin{equation*}
\int_0^{t_2}\mathcal{J}_1\, d\sigma \lesssim \langle t_2 \rangle^{\lceil \frac{1}{2}+\alpha \rceil+1}\|\langle x\rangle^{\frac{1}{2}+\alpha} \varphi\|_2.    
\end{equation*}
Gathering the previous estimates, we have proved
\begin{equation*}
\begin{aligned}
\|\mathcal{D}^{\frac{1}{2}+\alpha}\big(\mathcal{U}_{1,1}\psi\big)\|_2 \lesssim & \|\langle x\rangle^{\frac{1}{2}+\alpha} \varphi\|_2,
\end{aligned}    
\end{equation*}
thus, completing the considerations for $\mathcal{U}_{1,1}$. On the other hand, we write
\begin{equation*}
\begin{aligned}
\widehat{\varphi}(\xi)-\widehat{\varphi}(0)=\xi \int_0^1\partial_{\xi} \widehat{\varphi}(\sigma \xi)\, d\sigma,   
\end{aligned}    
\end{equation*}
so that
\begin{equation*}
\begin{aligned}
\mathcal{U}_{1,3}= &-it_2\frac{1+(1-\alpha)|\xi|^{\alpha}}{(1+|\xi|^{\alpha})^2}\xi\psi\Big( \int_0^1\partial_{\xi} \widehat{\varphi}(\sigma \xi)\, d\sigma\Big).   
\end{aligned}    
\end{equation*}
Hence, using the linearity of the operator $D^{\frac{1}{2}+\alpha}$, property \eqref{pointwise2a}, and Proposition \ref{negativeleibnizprop}, we get
\begin{equation*}
\begin{aligned}
\|D^{\frac{1}{2}+\alpha}(\mathcal{U}_{1,3})\|_2\lesssim \int_0^1&\Big(\|(1+(1-\alpha)|\xi|^{\alpha})\xi \partial_{\xi} \widehat{\varphi}(\sigma\xi)\psi\|_2\\
&+\|(1+(1-\alpha)|\xi|^{\alpha})\xi \partial_{\xi}\widehat{\varphi}(\sigma\xi)\psi \mathcal{D}^{\frac{1}{2}+\theta}(|\xi|^{\alpha}\widetilde{\psi})\|_2\\ &+\|\mathcal{D}^{\frac{1}{2}+\theta}\big((1+(1-\alpha)|\xi|^{\alpha})\xi \partial_{\xi}\widehat{\varphi}(\sigma\xi)\psi\big)\|_2\Big)\, d\sigma.   
\end{aligned}    
\end{equation*}
Thus, the above inequality makes it clear that we can proceed as in the estimates for $\mathcal{J}_1$ above to get
\begin{equation*}
\begin{aligned}
\|D^{\frac{1}{2}+\alpha}(\mathcal{U}_{1,3})\|_2\lesssim  \|\langle x\rangle^{\frac{3}{2}+\alpha} \varphi \|_2.
\end{aligned} 
\end{equation*}
We remark that the integration on $L^2$ and a change of variable show that the scaling by $\sigma$ above does not affect our estimates. Collecting the previous results, we have completed the estimate for $\mathcal{U}_{1,m}$, $m=1,2,3$. 

Next, we deal with $\mathcal{U}_{4,2,m}$, $m=1,2,3$. The linearity of the operator $D^{\frac{1}{2}+\alpha}$ yields
\begin{equation*}
 \begin{aligned}
\sum_{m=1}^3\|D^{\frac{1}{2}+\alpha}&(\mathcal{U}_{4,2,m})\|_2 \\
\lesssim &\int_0^{t_2}\|D^{\frac{1}{2}+\alpha}\big(\frac{1+(1-\alpha)|\xi|^{\alpha}}{(1+|\xi|^{\alpha})^2}(F(t_2-\tau,\xi)-1)\widehat{u^2}(\xi,\tau)\|_{2}\, d\tau\\
&+\int_0^{t_2}\|D^{\frac{1}{2}+\alpha}\big(\frac{1+(1-\alpha)|\xi|^{\alpha}}{(1+|\xi|^{\alpha})^2}(1-\psi)\widehat{u^2}(\xi,\tau)\|_{2}\, d\tau\\
&+\int_0^{t_2}\|D^{\frac{1}{2}+\alpha}\big(\frac{1+(1-\alpha)|\xi|^{\alpha}}{(1+|\xi|^{\alpha})^2}(\widehat{u^2}(\xi,\tau)-\widehat{u^2}(0,\tau))\psi\|_{2}\, d\tau,
 \end{aligned}   
\end{equation*}
where the implicit constant depends on $t_2$. Replacing $\widehat{\varphi}$ by $\widehat{u^2}$, we observe that the terms inside the integral in the inequality above correspond to $\mathcal{U}_{1,m}$, $m=1,2,3$. Thus, the estimates developed above for the factors $\mathcal{U}_{1,m}$ and \eqref{eqinterpnon1} establish
\begin{equation*}
 \begin{aligned}
\sum_{m=1}^3\|D^{\frac{1}{2}+\alpha}(\mathcal{U}_{4,2,m})\|_2 \lesssim & \int_0^{t_2}\|\langle x\rangle^{\frac{3}{2}+\alpha}u^2(\tau)\|_2 \, d\tau\\
 \lesssim & \sup_{t\in [0,T]}\big(\|\langle x \rangle^{\frac{(3+2\alpha)s}{4s-1}} u(t)\|_2^{\frac{4s-1}{2s}}\|u(t)\|_{s,2}^{\frac{1}{2s}}\big),
 \end{aligned}   
\end{equation*}
which is bounded provided that $u\in C([0,T];Z_{s,r})$ with $\max\{1,\frac{(3+2\alpha)s}{4s-1}\}< r<\frac{3}{2}+\alpha$, and the implicit constant depends on $t_2$. Finally, the estimate for $\mathcal{U}_{4,1}$ resembles that of $\mathcal{U}_3$. Indeed, following the argument in \eqref{eqU3}, we have
\begin{equation*}
\begin{aligned}
\mathcal{U}_{4,1}\lesssim & \langle t_2\rangle \int_0^{t_2}\|J_{\xi}^{\frac{1}{2}+\alpha}\big(F(t_2-\tau,\xi)\langle \xi \rangle^{1-\alpha}\partial_{\xi}\widehat{u^2}(\tau)\big)\|_2\, d\tau  \\
\lesssim & \langle t_2\rangle \int_0^{t_2}\|\langle x\rangle^{\frac{1}{2}+\alpha}e^{(t_2-\tau)A} J^{1-\alpha}(xu^2)(\tau)\|_2\, d\tau \\
\lesssim & \langle t_2\rangle^{\lceil \frac{1}{2}+\alpha\rceil+1} \int_0^{t_2}\|\langle x\rangle^{\frac{1}{2}+\alpha}J^{1-\alpha}(xu^2)(\tau)\|_2\, d\tau.
\end{aligned}   
\end{equation*}
To control the inequality above, we commute weights and derivatives to find
\begin{equation*}
\begin{aligned}
\|\langle x\rangle^{\frac{1}{2}+\alpha}J^{1-\alpha}(xu^2)\|_2\lesssim &\|[\langle x\rangle^{\frac{1}{2}+\alpha},J^{1-\alpha}](xu^2)\|_2+\|[J^{1-\alpha},\frac{x}{\langle x\rangle}](\langle x \rangle^{\frac{3}{2}+\alpha}u^2)\|_2 \\
&+ \|\frac{x}{\langle x\rangle}J^{1-\alpha}(\langle x \rangle^{\frac{3}{2}+\alpha}u^2)\|_2.   
\end{aligned}    
\end{equation*}
Since $\frac{1}{2}+\alpha, 1-\alpha \leq 1$, one can prove that $\mathcal{C}_1=[\langle x\rangle^{\frac{1}{2}+\alpha},J^{1-\alpha}]$, and $\mathcal{C}_2=[J^{1-\alpha},\frac{x}{\langle x\rangle}]$ define bounded operators in $L^2$. For example, one way to see this is using the Pseudo-Differential Calculus in \cite[Proof of Lemma 6]{KenigMartelRobbiano2011}, in which one can check that $\mathcal{C}_1,\mathcal{C}_2$ are determined by symbols in the class $\mathcal{S}^{0,0}$ and thus, (57) in \cite{KenigMartelRobbiano2011} shows the desired result. Consequently, by Lemma \ref{leibnizhomog} and Sobolev embedding, we get 
\begin{equation*}
\begin{aligned}
\|\langle x\rangle^{\frac{1}{2}+\alpha}J^{1-\alpha}(xu^2)\|_2\lesssim & \|J^{1-\alpha}(\langle x \rangle^{\frac{3}{2}+\alpha}u^2)\|_2\\
\lesssim & \|J^{1-\alpha}(\langle x \rangle^{\frac{3+2\alpha}{2}\cdot \frac{1}{2}}u)^2\|_2\\
\lesssim & \|\langle x \rangle^{\frac{3+2\alpha}{2}\cdot \frac{1}{2}}u\|_{L^{\infty}}\|J^{1-\alpha}(\langle x \rangle^{\frac{3+2\alpha}{2}\cdot\frac{1}{2}}u)\|_2\\
\lesssim & \|J^{1-\alpha}(\langle x \rangle^{\frac{3+2\alpha}{2}\cdot\frac{1}{2}}u)\|_2^2\\
\lesssim &\|\langle x \rangle^{\frac{s(3+2\alpha)}{4(s+\alpha-1)}}u\|_2^{\frac{2(s+\alpha-1)}{s}}\|J^{s}u\|_{2}^{\frac{2(1-\alpha)}{s}},
\end{aligned}    
\end{equation*}
where the last line above follows from interpolation. Thus, since $s>2(1-\alpha)$ implies that $\frac{s(3+2\alpha)}{4(s+\alpha-1)}<\frac{3}{2}+\alpha$, we get $\mathcal{U}_{4,1}$ is bounded. The proof is complete.
\end{proof}

\begin{proof}[Proof of Proposition \ref{propuniquecon2}] Since $\frac{1+(1-\alpha)|\xi|^{\alpha}}{(1+|\xi|^{\alpha})^2}\psi\in L^2(\mathbb{R})$, using Theorem \ref{theorem9b}, it is enough to show 
\begin{equation}\label{nonintegrab}
\begin{aligned}
\mathcal{D}^{\frac{1}{2}+\alpha} \Big(\frac{1+(1-\alpha)|\xi|^{\alpha}}{(1+|\xi|^{\alpha})^2}\psi \Big) \notin L^2(\mathbb{R}).  
\end{aligned}    
\end{equation}
We denote by $\mathcal{W}(x)=\frac{1+(1-\alpha)x}{(1+x)^2}$. We let $\mathcal{S}_1=\{\xi: 0<\xi<\frac{1}{2}\}$, and given $\xi \in \mathcal{S}_1$, we define $\mathcal{S}_2(\xi)=\{\eta : 0<\eta<\min\{(\frac{1}{2})^{\alpha},\frac{1}{100}\} \xi\}$. We note that if $\xi\in \mathcal{S}_1$, and $\eta\in \mathcal{S}_2(\xi)$,
\begin{equation*}
    |\xi-\eta|\sim |\xi|,
\end{equation*}
and given that $\alpha\in (0,1)$,
\begin{equation*}
    ||\xi|^{\alpha}-|\eta|^{\alpha}|\sim |\xi|^{\alpha}.
\end{equation*}
Moreover, by the mean value theorem, there exists some $z\in [|\eta|^{\alpha},|\xi|^{\alpha}]$ such that
\begin{equation*}
\begin{aligned}
|\mathcal{W}(|\xi|^{\alpha})-\mathcal{W}(|\eta|^{\alpha})|=&|\mathcal{W}'(z)|||\xi|^{\alpha}-|\eta|^{\alpha}|\\
=&\Big|\frac{(1+\alpha)+(1-\alpha)z}{(1+z)^3}\Big|||\xi|^{\alpha}-|\eta|^{\alpha}|\\
\gtrsim & ||\xi|^{\alpha}-|\eta|^{\alpha}|\\
\sim & |\xi|^{\alpha},
\end{aligned}    
\end{equation*}
where we have used that the restrictions on $\xi$ and $\eta$ in the sets $\mathcal{S}_1$ and $\mathcal{S}_2(\xi)$ imply $(1+z)^{3}\sim 1$. We emphasize that the implicit constants above are independent of $\xi \in \mathcal{S}_1$, and $\eta\in \mathcal{S}_2(\xi)$, and they depend on $\alpha\in (0,1)$. Now, from the definition of the fractional derivative, the fact that $\psi\equiv 1$ when $|\xi|\leq 1$, and the previous estimates, we have
\begin{equation*}
\begin{aligned}
\big(\mathcal{D}^{\frac{1}{2}+\alpha}(\mathcal{W}(|\xi|^{\alpha})\psi)\big)^2(\xi)\chi_{\mathcal{S}_1}(\xi)\geq & \Big(\int_{\mathcal{S_2}(\xi)} \frac{|\mathcal{W}(\xi)-\mathcal{W}(\eta)|^2}{|\xi-\eta|^{2+2\alpha}}\, d\eta \Big) \chi_{\mathcal{S}_1}(\xi)\\
\gtrsim & |\xi|^{2\alpha}\Big(\int_{\mathcal{S_2}(\xi)} \frac{1}{|\xi|^{2+2\alpha}}\, d\eta \Big) \chi_{\mathcal{S}_1}(\xi)\\
\gtrsim & |\xi|^{-1} \chi_{\mathcal{S}_1}(\xi).
\end{aligned}
\end{equation*}
Since $ |\xi|^{-\frac{1}{2}} \chi_{\mathcal{S}_1}\notin L^2(\mathbb{R})$, it follows from the previous inequality that \eqref{nonintegrab} holds.  
\end{proof}


\section*{Acknowledgments}
The authors thank the anonymous referee for valuable suggestions and comments. The authors were supported by Universidad Nacional de Colombia-Bogot\'a.



\begin{thebibliography}{99}
\small


\bibitem{ABFS}  L. Abdelouhab, J. L. Bona, M. Felland, and J.-C. Saut,  \emph{Nonlocal models for nonlinear dispersive waves}, Physica D. {\bf 40} (1989) 360--392.


\bibitem{An} J. Angulo, M. Scialom, and   C. Banquet, \emph{The regularized Benjamin-Ono and BBM equations: Well-posedness and nonlinear stability}, J. Diff. Eqs., \textbf{250} (2011), 4011--4036.

\bibitem{Angulo2018} J. Angulo Pava, \emph{Stability properties of solitary waves for fractional KdV and BBM equations}, Nonlinearity, {\bf 31} (3) (2018) 920--956.


\bibitem{AroSmit} N. Aronszajn and K. T. Smith, \emph{Theory of Bessel potentials}. I. Ann. Inst. Fourier (Grenoble) {\bf 11} (1961) 385--475.


\bibitem{BenjaminBonaMahony1972}  T. B. Benjamin, J. L. Bona, and J. J. Mahony \emph{Model equations for long waves in nonlinear dispersive systems}, Philosophical Transactions of the Royal Society of London. Series A, Mathematical and Physical Sciences {\bf 272} (1220) (1972) 47--78.


\bibitem{Be}  T. B. Benjamin, \emph{Internal waves of permanent form in fluids of great depth}, J. Fluid Mech. {\bf 29} (1967) 559--592.


\bibitem{KalischBona2000} J. L. Bona and H. Kalisch, \emph{Models for internal waves in deep water}, Discret. Contin. Dynam. Syst., {\bf6} (1) (2000) 1--20.

\bibitem{BoSm} J. L. Bona and R. Smith, \emph{The initial value problem for the Korteweg-de Vries equation}, Roy. Soc. London {\bf Ser A 278}  (1978) 555--601.


\bibitem{BoLi}  J. Bourgain and   D. Li., \emph{On an endpoint Kato-Ponce inequality}, Diff. and Int.  Eqs.  \textbf{27} (11-12) (2014),  1037--1072.
 
\bibitem{Ca} A. P. Calder\'on, \emph{Commutators of singular integral operators}, Proc. Nat. Acad. Sci. U.S.A., {\bf 53} (1965) 1092--1099.

\bibitem{CunhaRiano2022} A. Cunha and O. Ria\~no, \emph{The generalized fractional KdV equation in weighted Sobolev spaces}. Arxiv:2212.11160, 2022.

\bibitem{DaMcPo} L. Dawson, H. McGahagan, and G. Ponce,
\emph{On the decay properties of solutions to a class of Schr\"odinger equations}, Proc. AMS.  {\bf 136} (2008) 2081--2090.


\bibitem{FraLenz2013} R. L. Frank and E. Lenzmann, \emph{ Uniqueness of non-linear ground states for fractional Laplacians in $\mathbb{R}$},  Acta Math., {\bf 210} (2) (2013) 261--318.
 

\bibitem{FraLenzSilv2016} R. L. Frank, E. Lenzmann, and L. Silvestre, \emph{Uniqueness of radial solutions for the fractional Laplacian}, Comm. Pure Appl. Math., {\bf 69} (9) (2016) 1671--1726.

\bibitem{GFFLGP1}  G. Fonseca, F. Linares, and G. Ponce, \emph{The IVP for the dispersion generalized  Benjamin-Ono equation in weighted Sobolev spaces}, Ann. Inst. H. Poincar\'e Anal. Non Lin\'eaire,  \textbf{30} (5) (2013),  763--790.

\bibitem{FoPo} G. Fonseca and G. Ponce,\emph{ The I.V.P for the Benjamin-Ono equation in weighted
 Sobolev spaces},   J. Funct.  Anal. {\bf 260} (2011) 436--459.
       
\bibitem{FoRoSa} G. Fonseca, G. Rodriguez-Blanco, and W. Sandoval,\emph{ Well-posedness and ill-posedness results for the regularized  Benjamin-Ono equation in weighted Sobolev spaces},   Comm. Pure Appl. Math. {\bf 14} (4) (2015) 1327--1341.


\bibitem{GrafakosOh2014}  L. Grafakos and   S. Oh., \emph{The Kato-Ponce inequality}, Comm. Partial Differential Equations.  \textbf{39} (6) (2014),  1128--1157.


\bibitem{HerIonescuKenigKoch2010} S. Herr, A. D. Ionescu, C. E. Kenig, and H. Koch,  \emph{A Para-Differential Renormalization Technique for Nonlinear Dispersive Equations}, Comm. PDE, {\bf 35} (10) (2010) 1827--1875.

\bibitem{MuHuWh} R. Hunt, B. Muckenhoupt,  and R. Wheeden, \emph{Weighted norm inequalities for the conjugate function and  Hilbert transform}, Trans. AMS. {\bf 176} (1973) 227--251.


\bibitem{HunterIfrimTataruWong2015} J. K. Hunter, M. Ifrim, D. Tataru, and T. K. Wong, \emph{Long-time solutions for a Burgers-Hilbert equation via a modified energy method}, Proc. Amer. Math. Soc., {\bf 143} (8) (2015) 3407--3412.

\bibitem{Io1} R. J. I\'orio, \emph{On the Cauchy problem for the Benjamin-Ono equation},   Comm. P. D. E.  {\bf 11} (1986) 1031--1081.

\bibitem{Io2} R. J. I\'orio, \emph{Unique continuation principle for the Benjamin-Ono equation},   Diff. and Int. Eqs.  {\bf 16} (2003) 1281--1291.

\bibitem{Kal} H. Kalisch, \emph{Error analysis of a spectral projection of the regularized Benjamin-Ono equation}, \newblock  BIT, \textbf{45} (2005), 69--89.


\bibitem{KatPoc} T. Kato and G. Ponce, \emph{Commutator estimates and the Euler and Navier-Stokes equations}, Comm. Pure Appl. Math., {\bf 41} (7) (1988) 91--907.  

\bibitem{KenigMartelRobbiano2011}  C. E. Kenig, Y. Martel, and L. Robbiano. \emph{Local well-posedness and blow-up in the energy space for a class of $L^2$ critical dispersion generalized Benjamin-Ono equations}, Ann. Inst. H. Poincar\'e Anal. Non Lin\'eaire, {\bf 28} (6) (2011) 853--887.

\bibitem{KenigPonceVega1993}  C. E. Kenig, G. Ponce, and L. Vega., \emph{Well-posedness and scattering results for the generalized Korteweg-de Vries equation via the contraction principle}, Comm. Pure Appl. Math., {\bf 46} (4) (1993) 527--620.


\bibitem{KenigPilodPonceVega2020} C. E. Kenig, D. Pilod, G. Ponce, and L. Vega, \emph{On the unique continuation of solutions to non-local non-linear dispersive equations}, Comm. Partial Differential Equations, {\bf 45} (8) (2020) 872--886.

\bibitem{KleinSaut2015} C. Klein and J.-C. Saut, \emph{A numerical approach to blow-up issues for dispersive perturbations of {B}urgers' equation}, Phys. D, {\bf 295/296} (2015) 46--65.
         
\bibitem{KoTz1} H. Koch and  N. Tzvetkov, \emph{On the local well-posedness of the  Benjamin-Ono equation on $H^{s}(\mathbb{R})$}, Int. Math. Res. Not., {\bf 26} (2003)  1449--1464.

\bibitem{KoTz2} H. Koch and  N. Tzvetkov, \emph{Nonlinear wave interactions for the  Benjamin-Ono equation}, Int. Math. Res. Not., {\bf 30} (2005),  1833--1847.

\bibitem{Li} D. Li, \emph{On Kato-Ponce and Fractional Leibniz}, Rev. Mat. Iberoam., {\bf 35} (1) (2019) 23--100.

\bibitem{LiPiSa}  F. Linares, D. Pilod, and J.-C. Saut, \emph{Dispersive perturbations of Burgers and hyperbolic equations I: Local theory}, SIAM J. Math. Anal., {\bf 46} (2) (2014) 1505--1537.  

\bibitem{LinaresPilodSaut2015}  F. Linares,  D., Pilod, and J.-C. Saut \emph{Remarks on the orbital stability of ground state solutions of fKdV and related equations}, Adv. Differential Equations,  \textbf{20} (9/10) (2015),  835--858.

\bibitem{MeissHorton1982} J. Meiss and W. Horton, \emph{Fluctuation spectra of drift wave soliton gas}, Phys. Fluids, {\bf 25} (10) (1982), 1838--1843

\bibitem{Argenis2020} A. J. Mendez, \emph{On the propagation of regularity for solutions of the dispersion generalized Benjamin-Ono equation}, Anal. PDE, {\bf 13} (8) (2020) 2399--2440.

\bibitem{MolinetPilodVento2013}  L. Molinet, D. Pilod, and S. Vento, \emph{On well-posedness for some dispersive perturbations of Burgers' equation}, Ann. Inst. H. Poincar\'e Anal. Non Lin\'eaire, \textbf{35} (7) (2018),  1756--2018.

\bibitem{Mu} B. Muckenhoupt, \emph {Weighted norm inequalities for the Hardy maximal function}, Trans. AMS. {\bf 165} (1972) 207--226.

\bibitem{Na} J. Nahas, \emph{A decay property of solutions to the $k$-generalized KdV equation}, Adv. Differential Equations {\bf 17} (9/10) (2012)  833 -- 858.

\bibitem{NaPo} J. Nahas and G. Ponce, \emph{On the persistent properties of solutions to semi-linear Schr\"odinger equation}, Comm. P.D.E. {\bf 34} (2009) 1--20.


\bibitem{Nilsson2022} D. Nilsson, \emph{Extended lifespan of the fractional BBM equation}, Asymptot. Anal., {\bf 129} (2) (2022) 239--259.


\bibitem{OhWu2020} S. Oh and X. Wu, \emph{On $L^1$ endpoint Kato-Ponce inequality}, Math. Res. Lett., {\bf 27} (4) (2020) 1129--1163.

\bibitem{On} H. Ono, \emph{Algebraic solitary waves on stratified fluids},  J. Phy. Soc. Japan {\bf 39} (1975) 1082--1091.

\bibitem{Peregrine1966} D. H. Peregrine, \emph{Calculations of the development of an undular bore}, J. Fluid Mech. {\bf 25} (2) (1966) 321--330.

\bibitem{Pe} S. Petermichl, \emph{ The sharp bound for the Hilbert transform on weighted Lebesgue spaces in terms of the classical $A_p$ characteristic}, Amer. J. Math.   {\bf 129} (2007) 1355--1375.

\bibitem{Ponce} G. Ponce, \emph{Smoothing properties of solutions to the Benjamin-Ono equation}, In Analysis and partial differential equations, volume 122 of Lecture Notes in Pure and Appl. Math., pages 667-679. Dekker, New York, 1990.

\bibitem{Po} G. Ponce, \emph{On the global well-posedness of  the Benjamin-Ono equation}, Diff. \& Int. Eqs. {\bf 4} (1991) 527--542.

\bibitem{Ri} O. Ria\~no, \emph{On persistence properties in weighted spaces for solutions of the fractional Korteweg-de Vries equation}, Nonlinearity {\bf 34}, (7) (2021), 4604--4660.

\bibitem{Sa} J.-C. Saut, \emph{Sur quelques g\'en\'eralisations de l' \'equations de Korteweg-de Vries}, J. Math. Pures Appl. {\bf 58} (1979) 21--61.

\bibitem{St1} E. M. Stein,   \emph{The Characterization of Functions Arising as Potentials}, Bull. Amer. Math. Soc. {\bf 67} (1961) 102--104.


\bibitem{Ta} T. Tao, \emph{Global well-posedness of the   Benjamin-Ono equation on $H^{1}$}, Journal Hyp. Diff. Eqs. {\bf 1} (2004) 27--49.

\bibitem{Zeng2006} L. Zeng, \emph{Existence and stability of solitary-wave solutions of equations of Benjamin-Bona-Mahony type}, J. Differential Equations, {\bf 188} (1) (2003) 1--32
 
\end{thebibliography}
\end{document}